\documentclass{article}

\usepackage[top=1in, bottom=1.25in, left=1.4in, right=1.25in]{geometry}

\usepackage[utf8]{inputenc}
\usepackage[english]{babel}

\usepackage[breaklinks]{hyperref}
\hypersetup{%
  pdftitle={\@title},            
  pdfauthor={\@author},          
  unicode=true,           
  colorlinks=true,       
  linkcolor=blue,         
  citecolor=blue,      
  filecolor=green,        
  urlcolor=blue           
}

\usepackage[usenames,dvipsnames]{xcolor}

\usepackage[square,numbers]{natbib}

\usepackage{amsmath,amsfonts,amssymb,amsthm}
\usepackage{bm}
\usepackage{mathrsfs}
\usepackage{stmaryrd}

\newtheorem{theorem}{Theorem}
\newtheorem{proposition}{Proposition}
\newtheorem{lemma}{Lemma}

\usepackage{bm}

\global\long\def\Reals{\mathbb{R}}

\global\long\def\Nats{\mathbb{N}}
\global\long\def\Ints{\mathbb{Z}}
\global\long\def\NNInts{\Ints_{+}}
\global\long\def\NNReals{\Reals_{+}}

\global\long\def\intd{\mathrm{d}} 

\newcommand{\EE}{\mathbb{E}}

\newcommand{\PP}{\mathbb{P}}

\newcommand{\gammaDist}{\mathrm{Gamma}}
\newcommand{\betaDist}{\mathrm{Beta}}

\newcommand{\distiid}{\overset{iid}{\sim}}

\usepackage{mathtools}
\providecommand\given{} 
\newcommand\SetSymbol[1][]{
  \nonscript\,#1:\nonscript\,\mathopen{}\allowbreak}
\DeclarePairedDelimiterX\Set[1]{\lbrace}{\rbrace}%
{ \renewcommand\given{\SetSymbol[]} #1 }

\theoremstyle{plain}

\title{Optimal estimation of high-order missing masses, and the rare-type match problem}
\author{%
  Stefano Favaro\\
  University of Torino and Collegio Carlo Alberto\\
  10134 Torino, Italy\\
  \url{stefano.favaro@unito.it}
  \and
  Zacharie Naulet\\Université Paris-Saclay, CNRS, Laboratoire de mathématiques d’Orsay\\
  91405, Orsay, France\\
  \url{zacharie.naulet@universite-paris-saclay.fr}}
\date{}

\begin{document}

\maketitle

\begin{abstract}
For $n\geq1$, consider a random sample $(X_{1},\ldots,X_{n})$ from an unknown discrete distribution $P=\sum_{j\geq1}p_{j}\delta_{s_{j}}$ on a countable alphabet of symbols $\mathbb{S}$, and let $(Y_{n,j})_{j\geq1}$ be the empirical frequencies of distinct symbols $s_{j}$'s in the sample. In this paper, we consider the problem of estimating the $r$-order missing mass, which, for any $r\geq1$, is a discrete functional of $P$ defined as
\begin{displaymath}
\theta_{r}(P;\mathbf{X}_{n})=\sum_{j\geq1}p^{r}_{j}I(Y_{n,j}=0).
\end{displaymath}
This is generalization of the missing mass, or $1$-order missing mas, whose estimation is a classical problem in statistics, being the subject of numerous studies both in theory and methods. First, we introduce a nonparametric estimator of $\theta_{r}(P;\mathbf{X}_{n})$ and a corresponding non-asymptotic confidence interval through concentration properties of $\theta_{r}(P;\mathbf{X}_{n})$. Then, we investigate minimax estimation of $\theta_{r}(P;\mathbf{X}_{n})$ under a multiplicative or relative loss function, which is the main contribution of our work. We show that minimax estimation is not feasible over the class of all discrete distributions on $\mathbb{S}$, and not even for distributions with regularly varying tails, which only guarantee that our estimator is consistent for $\theta_{r}(P;\mathbf{X}_{n})$. This leads to introduce a stronger assumption for the tail behaviour of $P$, which is proved to be sufficient for minimax estimation of $\theta_r(P;\mathbf{X}_{n})$, making the proposed estimator an optimal minimax estimator of $\theta_{r}(P;\mathbf{X}_{n})$. Our interest in the $r$-order missing mass arises from forensic statistics, where the estimation of the $2$-order missing mass appears in connection to the estimation of the likelihood ratio $T(P,\mathbf{X}_{n})=\theta_{1}(P;\mathbf{X}_{n})/\theta_{2}(P;\mathbf{X}_{n})$, known as the rare-type match problem or the ``fundamental problem of forensic mathematics". We apply our results to the rare-type match problem, presenting the first theoretical guarantees to nonparametric estimation of $T(P,\mathbf{X}_{n})$.
\end{abstract}


\section{Introduction}

The estimation of the missing mass is a classical problem in statistics, dating back to the work of Alan M. Turing and Irving J. Good at Bletchley Park in 1940s \cite{Goo(53)}. Consider a population of units taking values in a (possibly infinite) universe $\mathbb{S}$ of symbols, i.e. a countable alphabet, and consider $n\geq1$ observable units from such a population. In its most common formulation, the problem of estimating the missing mass assumes that observable units are modeled as a random sample $\mathbf{X}_{n}=(X_{1},\ldots,X_{n})$ from an unknown distribution $P=\sum_{j\geq1}p_{j}\delta_{s_{j}}$, with $p_{j}$ being the probability of the symbol $s_{j}\in\mathbb{S}$, for $j\geq1$. Denoting by $(Y_{n,j})_{j\geq1}$ the empirical frequencies of distinct symbols in the sample, i.e. $Y_{n,j}=\sum_{1\leq i\leq n}I(X_{i}=j)$, the missing mass is defined as follows:
\begin{equation}\label{eq:missing}
\theta(P;\mathbf{X}_{n})=\sum_{j\geq1}p_{j}I(Y_{n,j}=0),
\end{equation}
namely the total probability mass of symbols not observed in the sample $\mathbf{X}_{n}$. The interest in the estimation of $\theta(P;\mathbf{X}_{n})$ has grown over the past three decades, primarily driven by biological and physical applications \cite{Mao(02),Gao(07),Ion(09),Dal(13),Dal(14)}. In biological sciences, the missing mass mostly appears as the probability of detecting unobserved genetic variants in new (unobservable) samples, which is a critical quantity to determine how many additional genomes must be sequenced in order to explain a certain proportion of genetic variation. See \cite{Den(19)}, and references therein, for an up-to-date overview on applications of the missing mass in biology. Other applications of the missing mass, as well as generalizations thereof, can be found in, e.g., statistical machine learning and information theory \cite{Bub(13),Orl(04),Ben(18)}, theoretical computer science \cite{Mot(06),Cai(18)}, empirical linguistics and natural language processing \cite{Gal(95)} and in forensic DNA analysis \cite{Cer(17)}.

The Good-Turing estimator is arguably the most popular estimator of the missing mass \cite{Goo(53),Rob(56),Rob(68)}. If $M_{n,r}$ denotes the number of distinct symbols with frequency $r\geq1$ in the random sample $\mathbf{X}_{n}$, i.e. $M_{n,r}=\sum_{j\geq 1}I(Y_{n,j}=r)$, then the Good-Turing estimator is
\begin{displaymath}
\hat{\theta}^{\text{\tiny{(GT)}}}(\mathbf{X}_{n})=\frac{M_{n,1}}{n}.
\end{displaymath}
This is a nonparametric estimator of $\theta(P;\mathbf{X}_{n})$, as its derivation does not rely on any assumption on the distribution $P$. In particular, $\hat{\theta}^{\text{\tiny{(GT)}}}(\mathbf{X}_{n})$ is obtained through a moment-based approach that compares the expected values of $\theta(P;\mathbf{X}_{n})$ and $M_{n,1}$ \cite{Goo(53)}. It also admits a nonparametric empirical Bayes derivation in sense of \cite{Rob(56),Rob(64)}, that is $\hat{\theta}(\mathbf{X}_{n})$ may be viewed as a posterior expectation with respect to an empirical (nonparametric) prior distribution. The Good-Turing estimator has been the subject of numerous studies, which led to comprehensive analysis of the problem of estimating the missing mass. These studies include, e.g., consistent and minimax estimation of $\theta(P;\mathbf{X}_{n})$ with respect to both a quadratic loss and a multiplicative loss function \cite{McA(00),Orl(03),Oha(12),Mos(15),Ach(18),Fad(18)}, large sample asymptotic properties of $\hat{\theta}^{\text{\tiny{(GT)}}}(\mathbf{X}_{n})$ in terms of central limit theorems, local limit theorem and sharp large deviations \cite{Est(82),Est(83),Zha(09), Gao13}, and non-asymptotic concentration properties of $\theta(P;\mathbf{X}_{n})$ with respect $\hat{\theta}^{\text{\tiny{(GT)}}}(\mathbf{X}_{n})$ \cite{Rob(68),McA(03), Oha(12), Ben(17),Fad(18)}.

\subsection{Our contributions}

In this paper, we consider the problem of estimating high-order missing masses, which generalize the missing mass $\theta(P;\mathbf{X}_{n})$ by taking a power function of order $r\geq1$ for the probabilities $p_{j}$'s in \eqref{eq:missing}. Formally, for $r\geq1$, we define the $r$-order missing mass as
\begin{equation} \label{eq:rmissing}
\theta_{r}(P;\mathbf{X}_{n})=\sum_{j\geq1}p^{r}_{j}I(Y_{n,j}=0).
\end{equation}
Clearly, the missing mass in \eqref{eq:missing} is recovered from \eqref{eq:rmissing} by setting $r=1$, i.e. the $1$-order missing mass. We introduce a nonparametric estimator of $\theta_r(P;\mathbf{X}_{n})$, which is of the form
\begin{displaymath}
\hat{\theta}_{r}^{\text{\tiny{(GT)}}}(\mathbf{X}_{n})= \frac{M_{n,r}}{{n\choose r}},
\end{displaymath}
and we obtain some (non-asymptotic) concentration properties of $\theta_r(P;\mathbf{X}_{n})$ with respect to $\mathbb{E}_{P}[\theta_r(P;\mathbf{X}_{n})]$. Confidence intervals for $\theta_r(P;\mathbf{X}_{n})$, with respect to the estimator $\hat{\theta}_{r}^{\text{\tiny{(GT)}}}(\mathbf{X}_{n})$, then follow as a corollary of our concentration inequalities. Our results do not rely on any assumption on the distribution $P$, and they generalize to the $r$-order missing mass some well-known concentration inequalities for the missing mass \cite{McA(03),Oha(12),Ben(17)}. Such a generalization is straightforward with respect to the left tail inequality, since it exploits the fact that $\theta_r(P;\mathbf{X}_{n})$ has a sub-Gaussian left tail for any $r\geq1$, as proved in \cite{McA(03),Ben(17)} for the missing mass $\theta(P;\mathbf{X}_{n})$. With respect to right tail, from \cite{Ben(17)} it is known that $\theta(P;\mathbf{X}_{n})$ has a sub-Gamma right tail, and our result leads to conjecture that such a tail behaviour is not preserved for $\theta_r(P;\mathbf{X}_{n})$ with $r\geq2$. For $r=1$, our concentration inequalities may be sharper than the corresponding inequalities in \cite{Ben(17)}. 

Then, we investigate minimax estimation of $\theta_{r}(P;\mathbf{X}_{n})$. Inspired by recent works on consistent estimation of the missing mass \cite{Oha(12),Mos(15), Fad(18)}, we consider a multiplicative or relative loss function. That is, if $\hat{\theta}_{r}(\mathbf{X}_{n})$ is an estimator of $\theta_{r}(P;\mathbf{X}_{n})$, we consider the loss
\begin{equation}\label{eq:80}
  \ell(\hat{\theta}_r(\mathbf{X}_{n}),\theta_r(P;\mathbf{X}_{n}))=\left|\frac{\hat{\theta}_r(\mathbf{X}_n)}{\theta_r(P;\mathbf{X}_n)} - 1\right|.
\end{equation}
The use of \eqref{eq:80} is motivated by the fact that $\theta_{r}(P;\mathbf{X}_{n})$ is a small-valued parameter, which makes more meaningful to measure errors in a relative loss rather than in an absolute distance loss \cite{Oha(12)}. With respect to the loss \eqref{eq:80}, we show that minimax estimation of $\theta_{r}(P;\mathbf{X}_{n})$ over the class $\mathcal{P}$ of all discrete distributions on $\mathbb{S}$ is not feasible, i.e.
\begin{displaymath}
  \forall \varepsilon \in (0,1),\qquad
\inf_{\hat{\theta}_{r}(\mathbf{X}_{n})}\sup_{P\in\mathcal{P}}\mathbb{P}_{P}(\ell(\hat{\theta}_r(\mathbf{X}_{n}),\theta_r(P;\mathbf{X}_{n}))\geq \varepsilon)=1.
\end{displaymath} 
This result leads to study conditions on $P$ to guarantee minimax estimation of $\theta_{r}(P;\mathbf{X}_{n})$. From \cite{Oha(12),Fad(18)} it is known that $\hat{\theta}^{\text{\tiny{(GT)}}}(\mathbf{X}_{n})$ is a consistent estimator of $\theta(P;\mathbf{X}_{n})$ if $P$ has regularly varying tails \cite{Bin(87),Gne(07)}, and here we show that an analogous result holds true for $\hat{\theta}_{r}^{\text{\tiny{(GT)}}}(\mathbf{X}_{n})$, for any $r\geq1$. In contrast, we prove that minimax estimation of $\theta_r(P;\mathbf{X}_{n})$ is not feasible under the assumption of regular variation, showing that the minimax rate over the class $\mathcal{P}$ is the same as the minimax rate over the class of all discrete distributions on $\mathbb{S}$ with regularly varying tails. Then, we introduce stronger assumptions for the tail behaviour of $P$, and we show that it is sufficient for minimax estimation of $\theta_r(P;\mathbf{X}_{n})$, making $\hat{\theta}_{r}^{\text{\tiny{(GT)}}}(\mathbf{X}_{n})$ an optimal estimator of $\theta_{r}(P;\mathbf{X}_{n})$. Such a result provides a solution to the open problem of optimality in minimax estimation of the missing mass $\theta(P;\mathbf{X}_{n})$, which was first discussed in \cite{Fad(18)}. Our proofs rely on novel Bayesian arguments, making use of suitable nonparametric priors that generate (almost surely) discrete distributions with regularly varying tails \cite{Pit(97),Gne(07)}.

Although the estimation of the $r$-order missing mass is of independent interest, our motivation to study such a problem arises from forensic statistics, where the estimation of the $2$-order missing mass appears in  the rare-type match problem, also known as the ``fundamental problem of forensic mathematics" \cite{Bre(10),Cer(17),Cer(20)}. The problem refers to situations where there is a match between the characteristics of some control material and the corresponding characteristics of the recovered material, with these characteristics being rare in the sense that they are not observed in any existing database of reference. The most popular example deals with a database of $n\geq1$ DNA profiles, say $\mathbf{X}_{n}$, and considers the following ``matching event": the suspect's DNA profile $X_{n+1}$ does not belong to the database $\mathbf{X}_{n}$ and $X_{n+1}$ matches the crime stain's DNA profile $X_{n+2}$. Assuming $\mathbf{X}_{n}$ to be a random sample from an unknown (discrete) distribution $P$, the estimation of the probability of the ``matching event" is applied to discriminate between the ``prosecution" hypothesis that the crime stain's profile comes from the suspect, i.e. $X_{n+1}$ is a random sample from $P$ and $X_{n+2}$ is equal to $X_{n+1}$ with probability one, and the ``defense" hypothesis that the crime stain's profile comes from an unknown donor, i.e. $X_{n+1}$ and $X_{n+2}$ are random samples from $P$. In particular, a largely accepted method consists in the estimation of the likelihood ratio or  evidence
\begin{displaymath}
T(P;\mathbf{X}_{n})=\frac{\mathbb{P}_{P}(\{X_{n+1}\notin\{X_{1},\ldots,X_{n}\}\}\cap\{X_{n+1}=X_{n+2}\};\,\text{``prosecution"})}{\mathbb{P}_{P}(\{X_{n+1}\notin\{X_{1},\ldots,X_{n}\}\}\cap\{X_{n+1}=X_{n+2}\};\,\text{``defense"})},
\end{displaymath}
where: i) the numerator of $T(P;\mathbf{X}_{n})$ is the probability of the ``matching event" under the ``prosecution" hypothesis, namely the missing mass $\theta(P;\mathbf{X}_{n})$; ii) the denominator of $T(P;\mathbf{X}_{n})$ is the probability of the ``matching event" under the ``defense" hypothesis, namely the $2$-order missing mass $\theta_{2}(P;\mathbf{X}_{n})$. We apply our results to the rare-type match problem, presenting the first theoretical guarantees to nonparametric estimation of $T(P,\mathbf{X}_{n})$.

\subsection{Organization of the paper}

The paper is structured as follows. In Section \ref{sec1} we introduce the estimator of the $r$-order missing mass $\hat{\theta}_{r}^{\text{\tiny{(GT)}}}(\mathbf{X}_{n})$ of $\theta_{r}(P;\mathbf{X}_{n})$, present a concentration inequality for $\theta_r(P;\mathbf{X}_{n})$, and apply this inequality to obtain a confidence interval for $\theta_r(P;\mathbf{X}_{n})$, with respect to $\hat{\theta}_{r}^{\text{\tiny{(GT)}}}(\mathbf{X}_{n})$. Section \ref{sec2} contains a minimax analysis for the problem of estimating $\theta_{r}(P;\mathbf{X}_{n})$, showing that, with respect to a multiplicative loss function, $\hat{\theta}_{r}^{\text{\tiny{(GT)}}}(\mathbf{X}_{n})$ is an optimal estimator of $\theta_r(P;\mathbf{X}_{n})$ if $P$ has second-order regularly varying tails. In Section \ref{sec3} we apply our results to the problem of estimating the likelihood ratio $T(P;\mathbf{X}_{n})$.


\section{Estimation of $\theta_{r}(P;\mathbf{X}_{n})$ and confidence intervals}\label{sec1}

For $n\geq1$ let $\mathbf{X}_{n}=(X_{1},\ldots,X_{n})$ be a random sample from an unknown distribution $P=\sum_{j\geq1}p_{j}\delta_{s_{j}}$ on $\mathbb{S}$, with both the probability masses $p_{j}$'s and the $\mathbb{S}$-valued atoms $s_{j}$'s being unknown. The actual values taken by the $X_{i}$'s is not relevant for the estimation of $\theta_{r}(P;\mathbf{X}_{n})$, and hence $\mathbb{S}$ is an arbitrary space, e.g. the set $[0,1]$. We denote by $(Y_{n,j})_{j\geq1}$ the empirical frequencies of distinct symbols in $\mathbf{X}_{n}$, i.e. $Y_{n,j}=\sum_{i=1}^{n}I(X_{i}=s_{j})$ with $\sum_{j\geq1}Y_{n,j}=n$, and by $M_{n,r}$ the number of distinct symbols with frequency $r\geq1$ in $\mathbf{X}_{n}$, i.e. 
\begin{displaymath}
M_{n,r}=\sum_{j\geq1}I(Y_{n,j}=r)
\end{displaymath}
with $\sum_{1\leq r\leq n}rM_{n,r}=n$. Moreover, let $C_{n,r}=\sum_{j\geq1}I(Y_{n,j}\geq r)$, such that $C_{n,1}$ is  the number of distinct symbols in $\mathbf{X}_{n}$. For any sequences $(a_n)_{n \geq 1}$ and $(b_n)_{n \geq 1}$, write $a_n\simeq b_n$ to mean that $a_{n}/b_{n}\rightarrow1$ as $n\rightarrow+\infty$. An estimator of the $r$-order missing mass $\theta_{r}(P;\mathbf{X}_{n})$ can be obtained through a moment-based approach. Specifically, we write
\begin{equation*}
\mathbb{E}_{P}(\theta_{r}(P;\mathbf{X}_{n}))= \sum_{j \geq 1} p_j^{r}(1-p_j)^{n}=\frac{1}{{n \choose r}} \sum_{j \geq 1} {n \choose r} p_j^{r}(1-p_j)^{n}\simeq \frac{\mathbb{E}_{P} (M_{n,r})}{{n \choose r}},
\end{equation*}
and set
\begin{displaymath}
\hat{\theta}_{r}^{\text{\tiny{(GT)}}}(\mathbf{X}_{n})=\frac{M_{n,r}}{{n\choose r}}
\end{displaymath}
as an estimator of $\theta_{r}(P;\mathbf{X}_{n})$, for any $n\geq1$. The estimator $\hat{\theta}_{r}^{\text{\tiny{(GT)}}}(\mathbf{X}_{n})$ is nonparametric, in the sense that the above moment-based derivation does not rely on any assumption on $P$.

Now, we obtain non-asymptotic concentration inequalities for $\theta_r(P;\mathbf{X}_{n})$ with respect to $\mathbb{E}_{P}[\theta_r(P;\mathbf{X}_{n})]$. Confidence intervals for $\theta_r(P;\mathbf{X}_{n})$ then follows as a corollary.

\begin{proposition}\label{pro:7}
  Let $v_n(P) \coloneqq \sum_{j\geq 1}p_j^{2r}(1-p_j)^n$. Then, for all $r \geq 1$, $n\geq 1$ and $x > 0$,
  \begin{equation}
    \label{eq:35}
    \PP_P\left( \theta_r(P;\mathbf{X}_n) - \EE_P[\theta_r(P;\mathbf{X}_n)] \leq - \sqrt{2v_n(P)x} \right)%
    \leq e^{-x},
  \end{equation}
  and, letting $c_0$ the unique solution to $\frac{\log(x)}{x} = -\frac{1}{2}$ over $\NNReals^{*}$ [$c_0 \approx 0.7$],
  \begin{equation}
    \label{eq:33}
    \PP_P\left( \theta_r(P;\mathbf{X}_n) - \EE_P[\theta_r(P;\mathbf{X}_n)] \geq \sqrt{2v_n(P)x} + \left[\frac{2\max(\frac{c_0}{2},\, x + \log(n))}{n} \right]^r \frac{2x}{3} \right)%
    \leq%
    2e^{-x}.
  \end{equation}
\end{proposition}

See Appendix~\ref{sec:proof-pro:tail-bounds} for the proof of Proposition \ref{pro:7}. The concentration inequalities of Proposition \ref{pro:7} do not rely on any assumption on the unknown distribution $P$, and they generalize to the $r$-order missing mass some concentration inequalities for the missing mass obtained in \cite{McA(03),Oha(12),Ben(17)}.  Regarding the left tail inequality \eqref{eq:35}, it follows by generalizing the proof of \cite[Proposition 3.7]{Ben(17)} to an arbitrary $r\geq1$. This is because $\theta_r(P;\mathbf{X}_n)$ is sub-Gaussian on the left tail, which allows to derive \eqref{eq:35} by adapting to the case $r\geq2$ the arguments applied in the proof of \cite[Proposition 3.7]{Ben(17)}. Regarding the right tail inequality \eqref{eq:33}, \cite[Theorem 3.9]{Ben(17)} established that the missing mass $\theta(P;\mathbf{X}_n)$ is sub-Gamma on the right tail, and then they obtained a concentration inequality through some standard arguments. See \cite[Chapter 2 and Chapter 3]{BLM(13)} and references therein. For $r \geq 2$ it is not clear whether $\theta_r(P;\mathbf{X}_n)$ is sub-Gamma on the right tail, and most likely it is not. The inequality \eqref{eq:33} can not be obtained by adapting to the case $r\geq2$ the arguments applied in the proof of \cite[Theorem 3.9]{Ben(17)}, and then it is obtained through a suitable truncation argument. For $r=1$, it is of interest to compare the results of Proposition \ref{pro:7} with the corresponding results in \cite[Proposition 3.7]{Ben(17)} and \cite[Theorem 3.9]{Ben(17)}. In particular, for the small $x \geq 0$ regime, our bounds have some advantages since the leading term $\sqrt{2v_n(P)x}$ is improved over \cite{Ben(17)}: it matches the left tail, and is the correct order of the variance of $\theta_r(P;\mathbf{X}_n)$. For the large $x \geq 0$ regime, however, we pay an extra $\log(n)$ factor that makes the bound of \cite{Ben(17)} bound better than our bound.

The next proposition exploits Proposition \ref{pro:7} in order to build a non asymptotic confidence interval for $\theta_r(P;\mathbf{X}_n)$, providing a way to quantify its uncertainty of the estimator $\hat{\theta}_{r}^{\text{\tiny{(GT)}}}(\mathbf{X}_{n})$. We define the lower bound of our confidence intervals as, for all $x > 0$,
\begin{equation}
  \label{eq:ic:lower}
  \mathfrak{L}_{n,r}(x)%
  \coloneqq%
  \hat{\theta}_{r}^{\text{\tiny{(GT)}}}(\mathbf{X}_{n})%
  - \frac{a_1(n,r)}{n^r} \sqrt{C_{n,r}x} - \frac{a_2(n,r)}{n^r} x - \frac{a_3(n,r)}{n^{1+r}} C_{n,r},
\end{equation}
where
\begin{gather*}
  a_1(n,r)%
  \coloneqq%
  n^r
  \Bigg[\sqrt{\frac{2}{\binom{n}{2r}}}%
  + \frac{2\sqrt{2}}{\binom{n}{r}}%
  \Bigg],\quad
  a_2(n,r)%
  \coloneqq%
  n^r\Bigg[ \frac{4}{\sqrt{\binom{n}{2r}}}%
  + \frac{8+2/3}{\binom{n}{r}}%
  + \frac{16 r }{\binom{n}{r+1}}%
  \Bigg],\quad
  a_3(n,r)%
  \coloneqq%
  n^{1+r}\frac{2r}{\binom{n}{r+1}}.
\end{gather*}
Along similar lines, we define the upper bound of our confidence intervals as, for all $x > 0$,
\begin{equation}
  \label{eq:ic:upper}
  \mathfrak{U}_{n,r}(x)%
  \coloneqq%
  \hat{\theta}_{r}^{\text{\tiny{(GT)}}}(\mathbf{X}_{n})%
    + \frac{b_1(n,r)}{n^r} \sqrt{C_{n,r}x} + \frac{b_2(n,r)}{n^r} x + \frac{[x+\log(n)]^r}{n^r} \frac{2^{r+1}x}{3}
  \end{equation}
where
\begin{gather*}
  b_1(n,r)%
  \coloneqq n^r\Bigg[ \sqrt{\frac{2}{\binom{n}{2r}}} + \frac{2 \sqrt{2}}{\binom{n}{r}} \Bigg],\qquad
  b_2(n,r)%
  \coloneqq n^r\Bigg[\frac{4}{\sqrt{\binom{n}{2r}}} + \frac{8 + 2/3}{\binom{n}{r}}%
  \Bigg].
\end{gather*}
The next proposition applies the lower bound \eqref{eq:ic:lower} and the upper bound \eqref{eq:ic:upper} to provide a non asymptotic confidence intervals for $\theta_r(P;\mathbf{X}_n)$. Again, we stress the fact that such a confidence interval does not rely on any assumption on the unknown distribution $P$.

\begin{proposition}
  \label{pro:7-confidence}
  If $n > 2r$, for any $r\geq1$, then under $P$ with probability at least $1 - 6e^{-x}$
  \begin{equation*}
    \theta_r(P;\mathbf{X}_n)%
    \geq \max\Big(0,\, \mathfrak{L}_{n,r}(x) \Big)
  \end{equation*}
  and with probability at least $1 - 7e^{-x}$
  \begin{equation*}
    \theta_r(P;\mathbf{X}_n)%
    \leq \min\Big(\mathfrak{U}_{n,r}(x),\, 1\Big).
  \end{equation*}
\end{proposition}

See Appendix~\ref{sec:proof-pro:confidence} for the proof of Proposition \ref{pro:7-confidence}. The numbers $a_1(n,r)$, $a_2(n,r)$, $a_3(n,r)$, $b_1(n,r)$ and $b_2(n,r)$ can also be traded by their asymptotic limits for more convenience. In particular, by means of Stirling's formula, a direct computation shows that
\begin{equation*}
  \lim_{n\to \infty}a_1(n,r)%
  =\lim_{n\to \infty}b_1(n,r)%
  = \sqrt{2}\Big(2r! + \sqrt{(2r)!}\Big),
\end{equation*}
and
\begin{equation*}
  \lim_{n\to \infty}a_2(n,r)%
  = \lim_{n\to\infty}b_2(n,r)
  = \frac{26}{3}r! + 4\sqrt{(2r)!},\quad%
  \lim_{n\to \infty}a_3(n,r)%
  = 2r (r+1)!.
\end{equation*}
The constants $a_j(n,r)$, $j=1,2,3$, and $b_j(n,r)$, $j=1,2$ are slightly over pessimistic, as our bounds are obtained by controlling separately $\theta_r(P;\mathbf{X}_n) - \EE_P(\theta_r(P;\mathbf{X}_n))$ and $\hat{\theta}_{r}^{\text{\tiny{(GT)}}}(\mathbf{X}_{n}) - \EE_P[\hat{\theta}_{r}^{\text{\tiny{(GT)}}}(\mathbf{X}_{n})]$. In contrast, for $r=1$ \cite[Proposition 5.5.]{Ben(17)} builds a confidence interval by directly controlling  $\theta_1(P;\mathbf{X}_n) - \hat{\theta}^{\text{\tiny{(GT)}}}(\mathbf{X}_{n})$ through a Poisson embedding argument \cite{Gne(07)}, i.e. assuming $n$ to be a Poisson random variable. The resulting interval is tight, and although it is claimed that a similar interval can be obtained without poissonization, its tightness is not obvious. One can show that $\theta_1(P;\mathbf{X}_n) - \hat{\theta}^{\text{\tiny{(GT)}}}(\mathbf{X}_{n}) = \sum_{j\geq 1}f_j(Y_{n,j})$ for some functions $(f_j)_{j\geq 1}$. Under the Poisson embedding the $Y_{n,i}$'s are independent random variables,  enabling for sharp concentration inequalities for $\theta_1(P;\mathbf{X}_n) -  \hat{\theta}^{\text{\tiny{(GT)}}}(\mathbf{X}_{n})$. Without the Poisson embedding, the $Y_{n,i}$'s are not independent random variables, and concentration inequalities can be obtained by relying on negative association. This requires to  decompose $f_j(Y_{n,j}) = g_j(Y_{n,j}) + h_j(Y_{n,j})$ where $(g_j(Y_{n,j}))_{j=1}^{n}$  are negatively associated (respectively $(h_j(Y_{n,j}))_{j=1}^n$), and then controlling separately $\sum_{j\geq 1}g_j(Y_{n,j})$ and $\sum_{j\geq 1}h_j(Y_{n,j})$.  Here there is no obvious decomposition that would improve over controlling separately $\theta_r(P;\mathbf{X}_n) - \EE_P(\theta_r(P;\mathbf{X}_n))$ and $\hat{\theta}_{r}^{\text{\tiny{(GT)}}}(\mathbf{X}_{n}) - \EE_P[\hat{\theta}_{r}^{\text{\tiny{(GT)}}}(\mathbf{X}_{n})]$, even when $r=1$. Another issue is that we do not have access to tight bounds on the variance of $M_{n,r}$, and hence of $\theta_r(P;\mathbf{X}_n)$. These issues disappear under the Poisson embedding, so that tighter intervals can be obtained.


\section{Optimal estimation of $\theta_{r}(P;\mathbf{X}_{n})$}\label{sec2}

We consider the problem of minimax estimation of the $r$-order missing mass $\theta_{r}(P;\mathbf{X}_{n})$. Inspired by some recent works on consistent estimation of the missing mass \cite{Mos(15), Fad(18)}, for an estimator $\hat{\theta}_{r}(\mathbf{X}_{n})$ of $\theta_{r}(P;\mathbf{X}_{n})$, we consider the multiplicative or relative loss function
\begin{align}
  \label{eq:801}
  \ell(\hat{\theta}_r(\mathbf{X}_{n}),\theta_r(P;\mathbf{X}_{n}))
 \coloneqq%
      \begin{cases}
    \Big|\frac{\hat{\theta}_r(\mathbf{X}_n)}{\theta_r(P;\mathbf{X}_n)} - 1\Big| &\mathrm{if}\quad\theta_r(P;\mathbf{X}_n) > 0,\\[0.2cm]
    h(\hat{\theta}_r(\mathbf{X}_n)) &\mathrm{otherwise},
  \end{cases}
\end{align}
where $h$ is any real-valued measurable function. The small values of $\theta_r(P;\mathbf{X}_n)$ makes more meaningful to measure the error of $\hat{\theta}_r(\mathbf{X}_{n}),$ through a loss function based on a relative distance, as \eqref{eq:801}, rather than through a loss function based on the absolute distance $|\hat{\theta}_r(\mathbf{X}_n) - \theta_r(P;\mathbf{X}_n)|$. See \cite{Oha(12),Mos(15)} for a detailed discussion. To see why, one may consider the trivial estimator $\hat{\theta}_r(\mathbf{X}_{n}):= 0$ for which $|\hat{\theta}_r(\mathbf{X}_n) - \theta_r(P;\mathbf{X}_n)| = \theta_r(P;\mathbf{X}_n) = o_P(1)$ as $n \to \infty$. In general, there may be several  unreasonable estimators of $\theta_{r}(P;\mathbf{X}_{n})$ that are consistent in absolute distance loss. The next theorem shows that minimax estimation of $\theta_r(P;\mathbf{X}_n)$ over the class $\mathcal{P}$ of all discrete distributions on $\mathbb{S}$ is not feasible. 

\begin{theorem}
  \label{thm:3}
  For every $r \geq 1$, $n \geq 2$, and $\varepsilon \in (0,1)$
  \begin{equation*}
    \inf_{\hat{\theta}_r(\mathbf{X}_{n})}\sup_{P\in \mathcal{P}}\PP_P\left( \ell(\hat{\theta}_r(\mathbf{X}_{n}),\theta_r(P;\mathbf{X}_{n})) > \varepsilon \right) =1.
  \end{equation*}
\end{theorem}

See Appendix~\ref{sec:proof-theor-refthm:3} for the proof of Theorem~\ref{thm:3}. Though the purpose of this paper is to derive non asymptotic (and quantitative) minimax bounds, it makes sense to compare the result of Theorem~\ref{thm:3} with the asymptotic results established in \cite{Mos(15)} in the case $r=1$. Following \cite{Mos(15)} we define the following criterion of asymptotic performance:
  \begin{itemize}
    \item \textit{Weak consistency}. A sequence $(\hat\theta_r(\mathbf{X}_n))_{n\geq 1}$ of $\bm{X}_n$-measurable estimators is weakly consistent relative to $\mathcal{P}$ if for all $\varepsilon \in (0,1)$, $ \sup_{P\in \mathcal{P}}\limsup_{n\to \infty} \PP_P\left( \ell(\hat{\theta}_r(\mathbf{X}_{n}),\theta_r(P;\mathbf{X}_{n})) >  \varepsilon \right) = 0$.
    \item \textit{Strong consistency}. A sequence $(\hat\theta_r(\mathbf{X}_n))_{n\geq 1}$ of $\bm{X}_n$-measurable estimators is strongly consistent relative to $\mathcal{P}$ if for all $\varepsilon \in (0,1)$, $\limsup_{n\to \infty} \sup_{P\in \mathcal{P}}\PP_P\left( \ell(\hat{\theta}_r(\mathbf{X}_{n}),\theta_r(P;\mathbf{X}_{n})) >   \varepsilon \right) =0$.
  \end{itemize}
  In particular \cite{Mos(15)} establish the striking result that there is no weakly consistent relative to $\mathcal{P}$ sequence of estimators of the missing mass.  This also implies that there is no strongly consistent sequence. Our Theorem~\ref{thm:3} trivially implies that for all values of $r\geq 1$ there is no strongly consistent relative to $\mathcal{P}$ sequence of estimators of the $r$-order missing mass. As pointed out in \cite[Section~3.1]{Mos(15)} proving the impossibility of strong consistency is rather trivial thanks to the following rationale. Consider two distributions $P_j = (1- \omega_j)\delta_{\heartsuit} + \omega_j \delta_{\diamond}$, $\omega_j$ being small positive numbers, $j=1,2$. By taking $\omega_1,\omega_2$ very small the event $E \coloneqq \Set{X_1=\dots=X_n = \heartsuit}$ has probability $\approx 1$ under both $P_1$ and $P_2$. Thus on $E$ any estimator $\hat\theta_r(\mathbf{X}_n)$ will predict the same value for $\theta_r(P_j;\mathbf{X}_n)$ regardless if $j=1$ or $j=2$, but the $r$-order missing mass equals $\omega_{1}^r$ under $P_1$ and $\omega_2^r$ under $P_2$. In fact this rationale can be adapted using Le Cam's two point method to establish a quantitative lower bound on the minimax risk that is slightly weaker (but simpler) than our Theorem~\ref{thm:3}. For the sake of completeness, we establish this bound in Section~\ref{sec:proof-theor-refthm:3:alternative}. Given as such, our Theorem~\ref{thm:3} does not answer the very much interesting question of the impossibility of weak consistency relative to $\mathcal{P}$. But our proof, which uses the same ideas as \cite{Fad(18)}, relying on lower bounding the minimax risk by the Bayes risk relative to a Dirichlet Process prior on $\mathcal{P}$, can easily be adapted to establish the impossibility of weak consistency relative to $\mathcal{P}$, for all values of $r\geq 1$. More precisely, the following theorem (proven in Section~\ref{sec:proof-theor-weakconsistency}) can be established using the same arguments as in the proof of Theorem~\ref{thm:3}.%
  \begin{theorem}
    \label{thm:impossibility-weak-consistency}
    For every $r\geq 1$, every $\varepsilon \in (0,1)$, and every sequence $(\hat\theta_{r,n})_{n\geq 1}$ of $\mathbf{X}_n$-measurable estimators
    \begin{equation*}
      \sup_{P\in \mathcal{P}}\limsup_{n\to \infty}\PP_P\left( \ell(\hat{\theta}_{r,n}(\mathbf{X}_{n}),\theta_r(P;\mathbf{X}_{n})) >   \varepsilon \right) = 1
    \end{equation*}
  \end{theorem}
  See Appendix~\ref{sec:proof-theor-weakconsistency} for the proof of Theorem~\ref{thm:impossibility-weak-consistency}.

With respect to the multiplicative loss function \eqref{eq:801}, the main result of \cite{Oha(12)} shows that the Good-Turing estimator $\hat{\theta}^{\text{\tiny{(GT)}}}(\mathbf{X}_{n})$ is a consistent estimator of the missing mass $\theta(P;\mathbf{X}_{n})$ if $P$ has regularly varying tails. See also \cite{Mos(15)}. More precisely, a sufficient condition to enable estimation is to assume that the function $\bar{F}_P : [0,1] \to \NNInts$ such that
\begin{displaymath}
  \bar{F}_P(x) \coloneqq \sum_{j\geq 1}I(p_j > x)
\end{displaymath}
has regularly varying tails, namely there exists a tail-index $\alpha \in (0,1)$ and a slowly-varying function at infinity $L$, i.e.  $L : \NNReals \to \NNReals$ with $\lim_{t \to \infty} \frac{L(tz)}{L(t)} = 1$ for any $z>0$, such that
\begin{equation}
  \label{eq:29}
  \bar{F}_P(x) \sim x^{-\alpha}L(1/x)
\end{equation}
as $x \to 0$  \cite{Bin(87),Gne(07)}. We denote by $\Sigma(\alpha,L)$ the class of all discrete distributions on $\mathbb{S}$ that satisfy \eqref{eq:29}, i.e. the class of distributions with regularly varying tails. The next theorem generalizes the main result of \cite{Oha(12)} to the $r$-order missing mass. It shows that if $P$ belongs to $\Sigma(\alpha,L)$ then $\hat{\theta}_{r}^{\text{\tiny{(GT)}}}(\mathbf{X}_{n})$ is a consistent estimator of $\theta_{r}(P;\mathbf{X}_{n})$ for any $r\geq1$.

\begin{proposition}
  \label{pro:6}
  If $P \in \Sigma(\alpha,L)$, for some $\alpha \in (0,1)$ and slowly-varying at infinity function $L$, then
  \begin{displaymath}
    \sqrt{n^{\alpha}L(n)}\ell(\hat{\theta}_{r}^{\text{\tiny{(GT)}}}(\mathbf{X}_{n}),\theta_r(P;\mathbf{X}_n)) = O_P(1)
  \end{displaymath}
  as $n\to \infty$.
\end{proposition}

See Appendix~\ref{sec:proof-pro:6} for the proof of Proposition \ref{pro:6}. It is easy to see that Proposition \ref{pro:6} is true using well-known universal limit theorems for $M_{n,r}$ and $\theta_r(P;\mathbf{X}_n)$ under $P \in \Sigma(\alpha,L)$ \cite{Kar(67),Gne(07)}. These theorems guarantee that for any $P \in \Sigma(\alpha,L)$ we have $M_{n,r} \simeq \EE_P(M_{n,r})$ and $\theta_r(P;\mathbf{X}_n) \simeq \EE_P(\theta_r(P;\mathbf{X}_n))$ whenever $n$ is large enough, with the latter expectation being determined by $\alpha$ and $L$. From Proposition \ref{pro:6} one may ask whether there exists $n_{0}$ not depending on $P$ such that it is possible to estimate $\theta_r(P;\mathbf{X}_n)$ uniformly over the class $\Sigma(\alpha,L)$ at a given accuracy. The next theorem provides a negative answer to such a question, showing that the minimax rate of estimating $\theta_r(P;\mathbf{X}_n)$ over the class $\Sigma(\alpha,L)$ is the same as the minimax rate over the whole $\mathcal{P}$ when $0 < \alpha < 1$. That is, in the next theorem we show that minimax estimation of $\theta_r(P;\mathbf{X}_{n})$ over the class $\Sigma(\alpha,L)$ is not feasible.

\begin{theorem}
  \label{thm:4}
  For all $0 < \alpha < 1$, for all $L > 0$, and for all $\varepsilon \in (0,1)$
  \begin{equation*}
    \inf_{\hat{\theta}_r(\mathbf{X}_{n})}\sup_{P\in \Sigma(\alpha,L)}\PP_P\left(\ell(\hat{\theta}_r(\mathbf{X}_{n}),\theta_r(P;\mathbf{X}_{n})) \geq \varepsilon \right)
    = \inf_{\hat{\theta}_r(\mathbf{X}_{n})}\sup_{P\in \mathcal{P}}\PP_P\Big(\ell(\hat{\theta}_r(\mathbf{X}_{n}),\theta_r(P;\mathbf{X}_{n})) \geq \varepsilon \Big) =1.
  \end{equation*}
\end{theorem}

See Appendix~\ref{sec:proof-thm:4} for the proof of Theorem \ref{thm:4}. The proof of the Theorem \ref{thm:4} consists in showing that the class $\Sigma(\alpha,L)$ is dense in the class $\mathcal{P}$ for the topology induced by the total variation metric. This result is stated in the next proposition, and it may be of independent interest in the study of distributions with regularly varying tails.

\begin{proposition}
  \label{pro:1}
  For all $P \in \mathcal{P}$, for all $\alpha \in (0,1)$, for all $\varepsilon > 0$, and for all $L> 0$ there exists $Q \in \Sigma(\alpha,L)$ such that $\|P - Q\|_{\mathrm{TV}} \leq \varepsilon$.
\end{proposition}
See Appendix~\ref{sec:proof-prop-refpr-1} for the proof of Proposition~\ref{pro:1}.

In view of Theorem~\ref{thm:4}, minimax estimation of $\theta_r(P;\mathbf{X}_n)$ requires to assume more than  regularly varying tails for $\bar{F}_P$. As in the work of \cite{Fav(22)}, it is natural to consider classes of discrete distributions for which the variations of $\bar{F}_P$ near zero are well-controlled. Here, to make the analysis simpler, we shall restrict to distributions for which
\begin{displaymath}
  L_{\alpha}(P) \coloneqq \lim_{x\to 0}x^{\alpha}\bar{F}_P(x)  
\end{displaymath}
exists and is non-zero. See \cite{Fav(22)} for analogous tail assumptions on $P$. For $\alpha \in (0,1)$, $\beta > 0$, $C>0$ and $C' > 0$, the next theorem over the class
\begin{displaymath}
 \Sigma(\alpha,\beta,C,C')%
  \coloneqq%
  \Set*{P \in \mathcal{P} \given \sup_{x \in (0,1)}x^{-\beta}\Big|\frac{\bar{F}_P(x)}{L_{\alpha}(P)x^{-\alpha}} - 1  \Big| \leq C',\ C\leq  L_{\alpha}(P) < \infty   }.
\end{displaymath}
Of course, it is not necessary to assume that $P$ belongs to $\Sigma(\alpha,\beta,C,C')$ to establish minimax estimation of $\theta_r(P;\mathbf{X}_n)$, though it turns out to be a convenient sufficient condition.

\begin{theorem}
  \label{thm:5}
  For all $\alpha\in (0,1)$, for all $\beta > 0$, for all $C > 0$, for all $C' > 0$, for all $n\geq 3$, for all $1\leq r \leq n$, and for all $x > 0$, if $n \geq \max\{3,\, [\frac{12C'\Gamma(2r-\alpha+\beta+1)}{\Gamma(2r-\alpha)}]^{1/\beta}\}$ and $x \leq \frac{C(e/2)^{2r}}{96r^{1+\alpha}}n^{\alpha}$ then
  \begin{multline*}
     \sup_{P\in \Sigma(\alpha,\beta,C,C')}\PP_P\Bigg(\sqrt{n^{\alpha}L_{\alpha}(P) }\ell(\hat{\theta}_{r}^{\text{\tiny{(GT)}}}(\mathbf{X}_{n}),\theta_r(P;\mathbf{X}_n))\\  >
    16r(4e)^r\sqrt{\frac{1 + \frac{3C'\Gamma(r-\alpha+\beta)}{\Gamma(r-\alpha)n^{\beta}}}
      {\frac{\Gamma(r-\alpha)}{\Gamma(r)}}}\sqrt{x}%
    + \frac{8r(2e)^r(x + \log(n))^r}{\frac{\Gamma(r-\alpha)}{\Gamma(r)} \sqrt{C}n^{\alpha/2}}x\Bigg)
  \leq 7e^{-x}.
  \end{multline*}
  Moreover, for all $\alpha \in (0,1)$, for all $0 < \beta < \frac{\alpha}{2}$ there exists a constant $C(\alpha,\beta)$ depending solely on $(\alpha,\beta)$ such that for all $n \geq \max(4,\, (12\alpha)^{1/(1-\alpha)})$, for all $r\geq 1$, for all $C \leq \frac{1}{64\pi\Gamma(1-\alpha)}$ and for all $C' > C(\alpha,\beta)$
  \begin{displaymath}
    \inf_{\hat{\theta}_r(\mathbf{X}_{n})}\sup_{P\in \Sigma(\alpha,\beta,C,C')}\PP_P\left(\sqrt{n^{\alpha}L_{\alpha}(P) }\ell(\hat{\theta}_r(\mathbf{X}_n);\theta_r(P;\mathbf{X}_n))  > \sqrt{C}\min\Big(\frac{r}{8\sqrt{6\alpha}},\,\frac{n^{\alpha/2}}{2}\Big) \right)%
    \geq \frac{1}{4}.
  \end{displaymath}%
\end{theorem}

See Appendix~\ref{sec:proof-thm:5} for the proof of Theorem \ref{thm:5}. The proof of Theorem~\ref{thm:5} does not rely on the use of  Le Cam's two point method to establish the minimax lower bound. In fact, the use of Le Cam's argument requires a sharp bound on the total variation distance between the law of two random partitions with parameters $P_1,P_2 \in \Sigma(\alpha,\beta,C,C')$. See \cite[Proposition 1]{Fav(22)} for details.  It is easy to get such bound when one wants to consider the worst case with $P_1,P_2 \in \mathcal{P}$, certainly much more challenging when one wants to add the restriction that $P_1,P_2 \in \Sigma(\alpha,\beta,C,C')$. Instead, we lower bound the minimax risk by the Bayes risk for a suitable choice of prior distribution over $\Sigma(\alpha,\beta,C,C')$. The problem of optimality in minimax estimation of the missing mass $\theta(P;\mathbf{X}_{n})$ was first discussed \cite{Fad(18)}, where it was deferred to future work. For $r=1$, Theorem \ref{thm:5} provides a solution to such a problem.


\section{The rare-type match problem}\label{sec3}\label{sec3}

In forensic statistics, the rare-type match problem refers to the situation in which the suspect's DNA profile, matching the DNA profile that is found at the crime scene, is not in the database of reference, that is the DNA profile is considered to be a rare profile. Intuitively, the rarer the suspect's DNA profile the more guilty the suspect is. See \cite{Bre(10)}, and references therein, for a detailed account on the rare-type match problem and generalizations thereof. Within such a context, the crime stain's DNA profile serves as a piece of evidence to discriminate between the following mutually exclusive hypotheses:
\begin{itemize}
\item[i)] the ``prosecution" hypothesis that the crime stain's profile comes from the suspect; in other terms, the piece of evidence found on the crime scene is used against the suspect;
\item[ii)]  the ``defense" hypothesis that the crime stain's profile comes from an unknown donor; in other terms, the piece evidence found on the crime scene is used in favor of the suspect. 
\end{itemize}
Both the ``prosecution" hypothesis and the ``defense" hypothesis require to evaluate a matching probability that is defined as the probability that a DNA profile that is not in the database of reference is found on the crime scene and on the suspect. As the population of the DNA profiles is unknown, a common approach to evaluate the matching probability consists in estimating it from the available database \cite{Bre(10),Cer(17),Cer(20)}. The resulting estimates are then applied to evaluate the weight of the evidence with respect to the ``prosecution" hypothesis and the ``defense" hypothesis, leading to a decision. 

Following  \cite{Bre(10)}, the database of DNA profiles is modeled as a random sample $\mathbf{X}_{n}=(X_{1},\ldots,X_{n})$ from an unknown distribution $P=\sum_{j\geq1}p_{j}\delta_{s_{j}}$, with $p_{j}$ being the probability of the DNA profile $s_{j}$, for $j\geq1$. Then, the ``matching event" can be defined as follows: the suspect's DNA profile $X_{n+1}$ does not belong to the database $\mathbf{X}_{n}$ and $X_{n+1}$ matches the crime stain's DNA profile $X_{n+2}$. Under ``prosecution" hypothesis, the probability of the ``matching event" is evaluated assuming that $X_{n+1}$ is a random sample from $P$ and $X_{n+2}$ is equal to $X_{n+1}$ with probability one, which implies that
\begin{equation}\label{prob_p}
\mathbb{P}_{P}(\{X_{n+1}\notin\{X_{1},\ldots,X_{n}\}\}\cap\{X_{n+1}=X_{n+2}\};\,\text{``prosecution"})=\theta(P;\mathbf{X}_{n},\, \mathbf{X}_n).
\end{equation}
Instead, under the ``defense" hypothesis, the probability of the ``matching event" is evaluated assuming that both $X_{n+1}$ and $X_{n+2}$ are random samples from $P$, which implies that
\begin{equation}\label{prob_d}
\mathbb{P}_{P}(\{X_{n+1}\notin\{X_{1},\ldots,X_{n}\}\}\cap\{X_{n+1}=X_{n+2}\};\,\text{``defense"},\mathbf{X}_n)=\theta_{2}(P;\mathbf{X}_{n}).
\end{equation}
As pointed by \cite{Bre(10)}, it is assumed that the database is representative of all innocent suspects, which may fail to be true
in many situations. Yet, the problem described above is fundamental to understand because if we fail to give a proper analysis of DNA evidence under the simplest assumptions, it is impossible to give a proper analysis in any situation.

The estimation of the probabilities \eqref{prob_p} and \eqref{prob_d} is applied to discriminate between the ``prosecution" hypothesis and the ``defense" hypothesis. See \cite{Bre(10),Cer(17),Cer(20)} and references therein. In particular, a largely accepted method consists in the estimation of the likelihood ratio
\begin{displaymath}
T(P;\mathbf{X}_{n})=\frac{\theta_1(P;\mathbf{X}_{n})}{\theta_{2}(P;\mathbf{X}_{n})}.
\end{displaymath}
Hereafter, we show how the results of Section \ref{sec1} and Section \ref{sec2} apply to the estimation of $T(P;\mathbf{X}_{n})$. In particular, it follows that a natural nonparametric estimator of $T(P;\mathbf{X}_{n})$ is
\begin{equation*}
  \hat{T}(\mathbf{X}_n)%
  \coloneqq \frac{\hat{\theta}_1^{\text{\tiny{(GT)}}}(\mathbf{X}_n)}{\hat{\theta}^{\text{\tiny{(GT)}}}_2(\mathbf{X}_n)}.
\end{equation*}
Indeed, Proposition~\ref{pro:6}, in combination with Slutsky's lemma, is enough to guarantee that
\begin{equation*}
  \frac{\hat{T}(\mathbf{X}_n)}{T(P;\mathbf{X}_n)} = 1 + O_p\Big(\frac{1}{\sqrt{n^{\alpha}L(n)}}\Big)
\end{equation*}
if $P \in \Sigma(\alpha,L)$, which is an assumption considered plausible by \cite{Cer(17),Cer(20)}. Also of interest, our Proposition~\ref{pro:7-confidence} allows for building confidence intervals for $T(P;\cdot)$ which are non-asymptotic and valid for all $P \in \mathcal{P}$. Recalling the definitions of $\mathfrak{L}_{n,r}$ and $\mathfrak{U}_{n,r}$ from equations~\eqref{eq:ic:lower} and~\eqref{eq:ic:upper}, it is immediately deduce that Proposition~\ref{pro:7-confidence} that the intervals
\begin{equation*}
  \mathfrak{T}_n(x) \coloneqq%
  \begin{cases}
    [\frac{\max(0,\, \mathfrak{L}_{n,1}(x)}{\min(\mathfrak{U}_{n,2}(x),1)},\, \frac{\min(\mathfrak{U}_{n,1}(x),1)}{\mathfrak{L}_{n,2}(x)}] &\mathrm{if}\ \mathfrak{L}_{n,2}(x) > 0,\\[0.4cm]
    [\frac{\max(0,\, \mathfrak{L}_{n,1}(x)}{\min(\mathfrak{U}_{n,2}(x),1)},\, \infty) &\mathrm{otherwise},
  \end{cases}
\end{equation*}
have coverage probability at least $1 - 26e^{-x}$ under $P$, for all $P\in \mathcal{P}$. Although $\mathfrak{T}_n(x)$ might be vacuous if $P$ is not reasonable, this is not the case when $P$ has regularly varying tails. 

We conclude by establishing a minimax result analogous to Theorem~\ref{thm:5}, but for the ratio $T(P;\mathbf{X}_n)$. This result shows that second-order regular variation is sufficient for minimax estimation of $T(P;\mathbf{X}_n)$, making $\hat{T}(\mathbf{X}_{n})$ an optimal minimax estimator of $T(P;\mathbf{X}_n)$.

\begin{theorem}
  \label{thm:rare-types-minimax}
  For all $\alpha\in (0,1)$, for all $\beta > 0$, for all $C > 0$, for all $C' > 0$, for all $n\geq 3$, and for all $x > 0$, if $n \geq \max\{3,\, [\frac{12C'\Gamma(2r-\alpha+\beta+1)}{\Gamma(2r-\alpha)}]^{1/\beta}\}$ and $x \leq \min\{(\frac{Cn^{\alpha}}{6000})^{1/3},\,\frac{Cn^{\alpha}}{6000\log(n)^2},\, \frac{C n^{\alpha}}{5\cdot 10^8\sqrt{1 + 3e^2C'(\beta/n)^{\beta}}} \} $ then
  \begin{multline*}
    \sup_{P\in \Sigma(\alpha,\beta,C,C')}\PP_P\Bigg( \sqrt{n^{\alpha}L_{\alpha}(P) }\ell(\hat{T}(\mathbf{X}_{n}),T(P;\mathbf{X}_n))\\  > 12000\sqrt{1 + 3e^2C'\Big(\frac{\beta}{n} \Big)^{\beta}}\sqrt{x} + \frac{1500(x+\log(n))^2}{\sqrt{C}n^{\alpha/2}}x\Bigg)%
    \leq 14e^{-x}.
  \end{multline*}
  Moreover, for all $\alpha \in (0,1)$, for all $0 < \beta < \frac{\alpha}{2}$ there exists a constant $C(\alpha,\beta)$ depending solely on $(\alpha,\beta)$ such that for all $n \geq \max(4,\, (12\alpha)^{1/(1-\alpha)})$, for all $C \leq \frac{1}{64\pi\Gamma(1-\alpha)}$ and for all $C' > C(\alpha,\beta)$
  \begin{displaymath}
    \inf_{\hat{\theta}_r(\mathbf{X}_{n})}\sup_{P\in \Sigma(\alpha,\beta,C,C')}\PP_P\left(\sqrt{n^{\alpha}L_{\alpha}(P) }\ell(\hat{T}(\mathbf{X}_n);T(P;\mathbf{X}_n))  > \sqrt{C}\min\Big(\frac{r}{8\sqrt{6\alpha}},\,\frac{n^{\alpha/2}}{2}\Big) \right)%
    \geq \frac{1}{4}.
  \end{displaymath}%
\end{theorem}
See Appendix~\ref{sec:proof-theorem-rare-types-minimax} for the proof of Theorem~\ref{thm:rare-types-minimax}.


\appendix

\section{Proofs}
\label{sec:proofs-1}

\subsection{Proof of Proposition~\ref{pro:7}}
\label{sec:proof-pro:tail-bounds}

\subsubsection{Proof of the left tail inequality}

  Proceeding as in the proof of Proposition~3.7 in \cite{Ben(17)}, we have for all $\lambda \in \Reals$ that
  \begin{align*}
    \log \EE_P[e^{\lambda[\theta_r(P;\mathbf{X}_n) - \EE_P(\theta_r(;\mathbf{X}_n))]}]%
    &\leq \sum_{j\geq 1}(1 - p_j)^n\phi(\lambda p_j^r)
  \end{align*}
  with $\phi(x) \coloneqq e^x - 1 - x$. In particular, when $x \leq 0$ it is famous that $\phi(x) \leq x^2/2$. Therefore when $\lambda < 0$ it is the case that $\log\EE_P[e^{\lambda[\theta_r(P;\mathbf{X}_n) - \EE_P(\theta_r(;\mathbf{X}_n))]}] \leq \frac{\lambda^2}{2}v_n(P)$. The conclusion follows using Chernoff's bound.

\subsubsection{Proof of the right tail inequality}

  We let $\gamma > 0$ to be chosen accordingly later. We split $\theta_r(P;\mathbf{X}_n) - \EE_P[\theta_r(P;\mathbf{X}_n)]$ into two random variables, $Z_1 \coloneqq \sum_{j\geq 1}p_j^r[I(X_1\ne s_j,\dots,X_n\ne s_j) - (1-p_j)^n]I(p_j^r \leq \gamma)$, and $Z_2 \coloneqq \sum_{j\geq 1}p_j^r[I(X_1\ne s_j,\dots,X_n\ne s_j) - (1-P_j)^n]I(p_j^r > \gamma)$. We first obtain a right-tail inequality for $Z_1$. Proceeding as in the proof of Proposition~3.7 in \cite{Ben(17)} and via the classical argument that $x \mapsto \frac{e^x - x - 1}{x^2}$ is monotone increasing for $x > 0$, we have for all $\lambda > 0$
  \begin{align*}
    \log\EE_P[e^{\lambda Z_1}]%
    &\leq \sum_{j\geq 1}(1 - p_j)^n (\lambda P_j^{r})^2\frac{e^{\lambda p_j^r} - \lambda p_j^r - 1}{(\lambda p_j^r)^2}I(p_j^r \leq \gamma)\\
    &\leq \frac{\big(e^{\lambda \gamma} - \lambda \gamma - 1 \big)}{\gamma^2}%
      \sum_{j\geq 1} p_j^{2r}(1 - p_j)^n.
  \end{align*}
  Thus, $Z_1$ satisfies the following Bernstein inequality (see \cite{BLM(13)})
  \begin{equation}
    \label{eq:bernstein-z1}
    \PP_P\Big(Z_1 > \sqrt{2 v_n(P) x} + \frac{2}{3}\gamma x \Big) \leq e^{-x}.
  \end{equation}
  Now regarding $Z_2$, we see that almost-surely
  \begin{align*}
    Z_2%
    &\leq \sum_{j\geq 1}p_j^rI(X_1\ne s_j,\dots,X_n\ne s_j)I( p_j^r > \gamma)
  \end{align*}
  Therefore,
  \begin{align*}
    \PP(Z_2 > 0)%
    &\leq \PP\Bigg(\sum_{j\geq 1}p_j^rI(X_1\ne s_j,\dots,X_n\ne s_j)I( p_j^r > \gamma) > 0 \Bigg)\\
    &= \PP\Big( \exists j\geq 1,\ p_j^r > \gamma\ \mathrm{and}\ X_1 \ne s_j,\dots,X_n \ne s_j  \Big)\\
    &\leq \sum_{j\geq 1}\PP\big(X_1\ne s_j,\dots,X_n\ne s_j \big)I(p_j^r > \gamma\big)\\
    &= \sum_{j\geq 1}(1 - p_j)^nI(p_j^r > \gamma\big)\\
    &\leq \frac{1}{\gamma^{1/r}}\sum_{j\geq 1}p_j(1 - P_j)^nI(p_j^r > \gamma\big).
  \end{align*}
  Therefore,
  \begin{equation}
    \label{eq:conceq-z2}
    \PP(Z_2 > 0) \leq \frac{1}{\gamma^{1/r}}e^{-n\gamma^{1/r}}.
  \end{equation}
  Since $\Set{Z_1 + Z_2 > \sqrt{2v_n(P)x} + \frac{2}{3}\gamma x} \subset \Set{Z_1 > \sqrt{2v_n(P)x} + \frac{2}{3}\gamma x}\cup \Set{Z_2 > 0}$, we deduce by combining \eqref{eq:bernstein-z1} and \eqref{eq:conceq-z2} that
  \begin{equation*}
    \PP\big( \theta_r(P;\mathbf{X}_n) - \EE_P[\theta_r(P;\mathbf{X}_n)]> \sqrt{2v_n(P)x} + \frac{2}{3}\gamma x \big)%
    \leq e^{-x} + \frac{1}{\gamma^{1/r}}e^{-n\gamma^{1/r}}.
  \end{equation*}
  Then by choosing $n\gamma^{1/r} = 2\max\{\frac{c_0}{2},\,x + \log(n)\}$, it is found that
  \begin{align*}
    \PP\Big( \theta_r(P;\mathbf{X}_n) - \EE_P[\theta_r(P;\mathbf{X}_n)]> \sqrt{2v_n(P)x} + \frac{2}{3}\gamma x \Big)%
    &\leq e^{-x} + ne^{-n\gamma^{1/r}[1 + \frac{\log(n\gamma^{1/r})}{n\gamma^{1/r}}]}\\
    &\leq e^{-x} + ne^{-\frac{n\gamma^{1/r}}{2}}\\
    &\leq 2e^{-x}
  \end{align*}
  where the second line follows because $n\gamma^{1/r} \geq c_0$ implies $\frac{\log(n\gamma^{1/r})}{n\gamma^{1/r}} \geq -\frac{1}{2}$ [see that the function $x\mapsto \frac{\log(x)}{x}$ is monotone increasing on $\NNReals^{*}$ with $c_0$ the unique solution to $\frac{\log(x)}{x} = -\frac{1}{2}$], and where the third line follows because $\frac{n\gamma^{1/r}}{2} \geq x + \log(n)$.

\subsection{Proof of Proposition~\ref{pro:7-confidence}}
\label{sec:proof-pro:confidence}

\subsubsection{Intermediate useful results}

We recall the following result taken from Proposition~3.5 in \cite{Ben(17)}. Recall that $C_{n,r} \coloneqq \sum_{j=r}^nM_{n,r}$.
\begin{proposition}
  \label{pro:2}
  Let define $w_n(P) \coloneqq \min( \max_{k\in\Set{r,r+1}}k\EE_P(M_{n,k}) ,\EE_P(C_{n,r}))$. For all $x\geq 0$ and all $n > r$:
  \begin{equation*}
    \PP_P\Big( |M_{n,r} - \EE_P(M_{n,r})| \geq \sqrt{8w_n(P)x} + \frac{2x}{3} \Big) \leq 4e^{-x}.
  \end{equation*}
\end{proposition}

The following result can also be trivially obtained from Proposition~3.4 in \cite{Ben(17)}.
\begin{proposition}
  \label{pro:3}
  For all $x\geq 0$ and $n > r$:
  \begin{equation*}
    \PP_P\Big( C_{n,r} \leq \EE_P(C_{n,r}) - \sqrt{8\EE_P(C_{n,r})x} \Big)%
    \leq e^{-x},
  \end{equation*}
  and
  \begin{equation*}
    \PP_P\Big( C_{n,r} \geq \EE_P(C_{n,r}) + \sqrt{8\EE_P(C_{n,r})x} + \frac{4}{3}x \Big)%
    \leq e^{-x}.
  \end{equation*}
\end{proposition}

Before proving the bounds in Proposition~\ref{pro:7-confidence}, we state a certain number of intermediate results. First, Rearranging and manipulating the events involved in Proposition~\ref{pro:3}, we see that for all $x \geq 0$ and $n > r$
\begin{gather}
  \label{eq:30}
  \PP_P\Big(\sqrt{\EE_P(C_{n,r})} \leq \sqrt{C_{n,r}} + \sqrt{8x} \Big) \geq 1 - e^{-x}.
\end{gather}
By the same negative association argument as in the proof of Proposition~3.6 in \cite{Ben(17)}, we find that whenever $n \geq 2r$
\begin{align}
  \label{eq:42}
  v_n(P)%
  &\leq \sum_{j\geq 1}p_j^{2r}(1 - p_j)^n%
  \leq \sum_{j\geq 1}p_j^{2r}(1 - p_j)^{n-2r}%
    = \frac{\EE_P(M_{n,2r})}{\binom{n}{2r}}%
    \leq \frac{\EE_P(C_{n,r})}{\binom{n}{2r}}.
\end{align}
Also,
\begin{align*}
  \EE_P(\hat{\theta}_r^{\text{\tiny{(GT)}}}) -  \EE_P(\theta_r(P))%
  &= \sum_{j\geq 1}p_j^r(1-p_j)^{n-r}[1 - (1-p_j)^r]
\end{align*}
and then
\begin{align*}
  0 \leq%
  \EE_P(\hat{\theta}_r^{\text{\tiny{(GT)}}}) -  \EE_P(\theta_r(P))
  &\leq r\sum_{j\geq 1}p_j^{r+1}(1-p_j)^{n-r}\\
  &\leq \frac{r}{\binom{n}{r+1}}\EE_P(M_{n,r+1})
\end{align*}
from which we obtain
\begin{align}
  \label{eq:40}
  0%
  \leq \EE_P(\hat{\theta}_r^{\text{\tiny{(GT)}}}) - \EE_P(\theta_r(P))%
  \leq \frac{r}{\binom{n}{r+1}}\EE_P(C_{n,r}).
\end{align}

\subsubsection{Proof of right side bound}

Because $w_n(P) \leq \EE_P(C_{n,r})$, the Proposition~\ref{pro:2} and the equation~\eqref{eq:40} imply that we have with probability at least $1 - 4e^{-x}$
\begin{align*}
  \hat{\theta}_r^{\text{\tiny{(GT)}}}
  &\leq \EE_P(\hat{\theta}_r^{\text{\tiny{(GT)}}}) + \frac{1}{\binom{n}{r}}\Big( \sqrt{8\EE_P(C_{n,r})x} + \frac{2x}{3} \Big)\\
  &\leq \EE_P(\theta_r(P))%
    + \frac{r}{\binom{n}{r+1}}\EE_P(C_{n,r})%
    + \frac{1}{\binom{n}{r}}\Big( \sqrt{8\EE_P(C_{n,r})x} + \frac{2x}{3} \Big).
\end{align*}
So by Proposition~\ref{pro:7} and equation~\eqref{eq:42}, with probability at least $1 - 5e^{-x}$
\begin{align*}
  \hat{\theta}_r^{\text{\tiny{(GT)}}}%
  &\leq \theta_r(P)%
    + \sqrt{2v_n(P)x}
    + \frac{r}{\binom{n}{r+1}}\EE_P(C_{n,r})%
    + \frac{1}{\binom{n}{r}}\Big( \sqrt{8\EE_P(C_{n,r})x} + \frac{2x}{3} \Big)\\
  &\leq \theta_r(P)%
    + \Bigg[ \sqrt{\frac{2}{\binom{n}{2r}}}%
    + \frac{\sqrt{8}}{\binom{n}{r}}
    \Bigg] \sqrt{ \EE_P(C_{n,r}) x}%
    + \frac{r}{\binom{n}{r+1}}\EE_P(C_{n,r})%
    + \frac{1}{\binom{n}{r}}\frac{2x}{3}.
\end{align*}
Then deduce from equation~\eqref{eq:30} that with probability at least $1 - 6e^{-x}$
  \begin{align*}
    \theta_r(P)%
    &\geq \hat{\theta}_r^{\text{\tiny{(GT)}}}%
    - \Bigg[ \sqrt{\frac{2}{\binom{n}{2r}}}%
    + \frac{\sqrt{8}}{\binom{n}{r}}
    \Bigg]\Big(\sqrt{C_{n,r}} + \sqrt{8x} \Big) \sqrt{x}%
    - \frac{r}{\binom{n}{r+1}}\Big(\sqrt{C_{n,r}} + \sqrt{8x} \Big)^2%
    - \frac{1}{\binom{n}{r}}\frac{2x}{3}\\
    &\geq \hat{\theta}_r^{\text{\tiny{(GT)}}}%
    - \Bigg[ \sqrt{\frac{2}{\binom{n}{2r}}}%
    + \frac{\sqrt{8}}{\binom{n}{r}}
    \Bigg]\Big(\sqrt{C_{n,r}} + \sqrt{8x} \Big) \sqrt{x}%
    - \frac{2r}{\binom{n}{r+1}}\big(C_{n,r} + 8x \big)%
      - \frac{1}{\binom{n}{r}}\frac{2x}{3}\\
    &\geq%
      \hat{\theta}_r^{\text{\tiny{(GT)}}}%
      - \frac{2r}{\binom{n}{r+1}}C_{n,r}
      - \sqrt{C_{n,r}x}\Bigg[ \sqrt{\frac{2}{\binom{n}{2r}}}%
    + \frac{\sqrt{8}}{\binom{n}{r}} \Bigg]%
      - x\Bigg[ \frac{4}{\sqrt{\binom{n}{2r}}}%
      + \frac{8}{\binom{n}{r}} + \frac{16r}{\binom{n}{r+1}}%
      +\frac{1}{\binom{n}{r}}\frac{2}{3}
      \Bigg].
  \end{align*}

\subsubsection{Proof of left side bound}

Because $w_n(P) \leq \EE_P(C_{n,r})$, the Proposition~\ref{pro:2} and the equation~\eqref{eq:40} imply that we have with probability at least $1 - 4e^{-x}$
\begin{align*}
  \hat{\theta}_r^{\text{\tiny{(GT)}}}%
  &\geq \EE_P(\hat{\theta}_r^{\text{\tiny{(GT)}}}) - \frac{1}{\binom{n}{r}}\Big(\sqrt{8\EE_P(C_{n,r})x} + \frac{2x}{3}\Big)\\
  &\geq \EE_P(\theta_r(P)) - \frac{1}{\binom{n}{r}}\Big(\sqrt{8\EE_P(C_{n,r})x} + \frac{2x}{3}\Big).
\end{align*}
So by Proposition~\ref{pro:7} and equation~\eqref{eq:42}, with probability at least $1 - 6e^{-x}$
\begin{align*}
  \hat{\theta}_r^{\text{\tiny{(GT)}}}%
  &\geq \theta_r(P) - \sqrt{2v_n(P)x} - \Big[\frac{2\max(\frac{c_0}{2},\,x+\log(n))}{n}\Big]^r \frac{2x}{3} - \frac{1}{\binom{n}{r}}\Big(\sqrt{8\EE_P(C_{n,r})x} + \frac{2x}{3}\Big)\\
  &\geq \theta_r(P) - \sqrt{\frac{2 \EE_P(C_{n,r}) x}{\binom{n}{2r}}} -
   \Big[\frac{2\max(\frac{c_0}{2},\,x+\log(n))}{n}\Big]^r \frac{2x}{3} - \frac{1}{\binom{n}{r}}\Big(\sqrt{8\EE_P(C_{n,r})x} + \frac{2x}{3}\Big).
\end{align*}
  Then deduce from equation~\eqref{eq:30} that with probability at least $1 - 7e^{-x}$
\begin{align*}
  \theta_r(P)%
  &\leq \hat{\theta}_r^{\text{\tiny{(GT)}}}%
    + \sqrt{\frac{2x}{\binom{n}{2r}}}\big(\sqrt{C_{n,r}} + \sqrt{8x} \big) + \Big[\frac{2\max(\frac{c_0}{2},\,x+\log(n))}{n}\Big]^r \frac{2x}{3}
    + \frac{1}{\binom{n}{r}}\Big(\sqrt{8x}\big(\sqrt{C_{n,r}} + \sqrt{8x} \big) + \frac{2x}{3}\Big)\\
  &\leq \hat{\theta}_r^{\text{\tiny{(GT)}}}%
    + \sqrt{C_{n,r}x}\Bigg(\sqrt{\frac{2}{\binom{n}{2r}}} + \frac{2\sqrt{2}}{\binom{n}{r}} \Bigg)%
    + x\Bigg(\frac{4}{\sqrt{\binom{n}{2r}}} + \frac{8 + 2/3}{\binom{n}{r}} \Bigg)%
    + \Big[\frac{2\max(\frac{c_0}{2},\,x+\log(n))}{n}\Big]^r \frac{2x}{3}.
\end{align*}
Finally, note that it is assumed that $n > 2r$, which implies that $x + \log(n) > \log(2) > c_0/2$.

\subsection{Proof of Theorem~\ref{thm:3}}
\label{sec:proof-theor-refthm:3}

We use the same ideas as \cite{Fad(18)} relying on lower bounding the minimax risk by the Bayes risk relative to a Dirichlet Process prior on $\mathcal{P}$. Indeed, for the sake of simplifying future proofs, we derive all the computations under the \textit{Pitman-Yor} prior with parameters $(\alpha,d,G)$, $\alpha \in [0,1]$, $d > -\alpha$ to be chosen accordingly later, and $G$ a non-atomic center measure [which for the special case $\alpha=0$ coincides with the Dirichlet Process prior; we refer to \cite[Section~14.4]{Gho(17)} for more details and definitions].

We consider a Pitman-Yor process with parameters $(\alpha,d,G)$ as prior distribution, denoted in the sequel $\mathrm{PY}_{\alpha,d,G}$. The choice of $G$ is irrelevant for our purpose. Then, for any estimator $\hat{\theta}_r$ and any $\varepsilon \in (0,1)$
\begin{align}
  \label{eq:5bis}
  \sup_{P\in \mathcal{P}}\PP_P\Big( \ell(\hat{\theta}_r,\theta_r) > \varepsilon \Big)
    \geq \int_{\mathcal{P}}\PP_P\Big( \ell(\hat{\theta}_r,\theta_r) > \varepsilon \Big) \mathrm{PY}_{\alpha,d,G}(\intd P).
\end{align}
We now bound the rhs of the last display. We write $\Pi$ the joint distribution of $(\mathbf{X}_n,P)$ such that $\mathbf{X}_n\mid P \sim P^{\otimes n}$ with $P \sim \mathrm{PY}_{\alpha,d,G}$. It follows that,
\begin{align}
  \notag \int_{\mathcal{P}}\PP_P\Big(\ell(\hat{\theta}_r(\mathbf{X}_n),\theta_r(P;\mathbf{X}_n)) > \varepsilon \Big) \mathrm{PY}_{\alpha,d,G}(\intd P)
  \notag
  &= \EE_{\Pi}\big(I( \ell(\hat{\theta}_r(\mathbf{X}_n),\theta_r(P;\mathbf{X}_n)) > \varepsilon )\big)\\
  \notag
  &= \EE_{\Pi}\big[ \EE_{\Pi}\big(I(  \ell(\hat{\theta}_r(\mathbf{X}_n),\theta_r(P;\mathbf{X}_n)) > \varepsilon ) \mid \mathbf{X}_n\big) \big]\\
  \label{eq:11}
  &\geq \EE_{\Pi}\big[ \inf_{t \in \NNReals} \EE_{\Pi}\big(I(  \ell(t,\theta_r(P;\mathbf{X}_n)) > \varepsilon ) \mid \mathbf{X}_n\big) \big].
\end{align}
But by Lemma~\ref{lem:4}, conditional on $\mathbf{X}_n$ it is the case that
$\theta_r(P;\mathbf{X}_n)$ has the law of $YZ$ where $Y = W_0^{r}$ and
$Z = \sum_{j\geq 1}Q_j^r$ with $W_0 \mid \mathbf{X}_n \sim\mathrm{Beta}(d + \alpha K_n,  n - \alpha K_n)$
is independent of
$Q \mid \mathbf{X}_n \sim \mathrm{PY}_{\alpha,d+\alpha K_n,G}$. In particular $\theta_r(P;\mathbf{X}_n)$ is almost-surely non-zero, so $\ell(t,\theta_r(P;\mathbf{X}_n)) = \big|\frac{t}{\theta_r} -1 \big|$ almost-surely too. Hence by conditional independence of $Y$ and $Z$ and because $Y,Z > 0$ almost-surely:
\begin{align*}
  \inf_{t\in \NNReals}\EE_{\Pi}\big(I( \ell(t,\theta_r(P;\mathbf{X}_n)) > \varepsilon ) \mid \mathbf{X}_n\big)%
  \notag
  &= \inf_{t\in \NNReals}\Bigg(1 - \PP_{\Pi}\Big(\Big|\frac{t}{YZ} - 1 \Big| \leq \varepsilon  \mid \mathbf{X}_n\Big) \Bigg)\\
  \notag
  &= \inf_{t\in \NNReals}\Bigg(1 - \PP_{\Pi}\Big( \frac{t/Z}{1+\varepsilon} \leq Y \leq \frac{t/Z}{1-\varepsilon}  \mid \mathbf{X}_n\Big) \Bigg)\\
  \notag
  &=\inf_{t\in \NNReals}\Bigg(1 - \EE_{\Pi}\Big[\PP_{\Pi}\Big( \frac{t/Z}{1+\varepsilon} \leq Y \leq \frac{t/Z}{1-\varepsilon}  \mid \mathbf{X}_n,Z\Big) \mid \mathbf{X}_n\Big] \Bigg)\\
  \notag
  &\geq \inf_{t\in \NNReals}\Bigg(1 - \EE_{\Pi}\Big[\sup_{z > 0}\PP_{\Pi}\Big( \frac{t/z}{1+\varepsilon} \leq Y \leq \frac{t/z}{1-\varepsilon}  \mid \mathbf{X}_n,Z\Big) \mid \mathbf{X}_n\Big] \Bigg)\\
  \notag
  &= 1 - \sup_{t\in \NNReals}\PP_{\Pi}\Big(\frac{t}{1+\varepsilon} \leq Y \leq \frac{t}{1-\varepsilon} \mid \mathbf{X}_n \Big)\\
  &= 1 - \sup_{t\in \NNReals}\PP_{\Pi}\Big(\frac{t}{(1+\varepsilon)^{1/r}} \leq W_0 \leq \frac{t}{(1-\varepsilon)^{1/r}} \mid \mathbf{X}_n \Big).
\end{align*}%
We deduce from Lemma~\ref{lem:beta-bayes} with $a \equiv d + \alpha K_n$, $b \equiv n - \alpha K_n$, $\delta_1 \equiv (1+\varepsilon)^{-1/r}$, and $\delta_2 \equiv (1-\varepsilon)^{-1/r}$, that on the event where $\alpha K_n < n-1$ [recall $n\geq 2$ by assumption]
\begin{align*}
  \notag
  \inf_{t\in \NNReals}\EE_{\Pi}\big(I( \ell(t,\theta_r(P;\mathbf{X}_n)) > \varepsilon ) \mid \mathbf{X}_n\big)%
  &\geq 1 - \Bigg(\frac{(1+\varepsilon)^{1/r}}{(1-\varepsilon)^{1/r}} - 1 \Bigg)\sqrt{\frac{(d + \alpha K_n)(d + n-1)}{2\pi(n - \alpha K_n - 1)}}e^{\frac{1}{12(d+n-1)}}\\
  &\geq 1 - \frac{\varepsilon/r}{1-\varepsilon} \sqrt{\frac{(d + \alpha K_n)(d + n-1)}{n - \alpha K_n - 1}}
\end{align*}
Therefore going back to the estimates in equations~\eqref{eq:5bis} and~\eqref{eq:11}, we find that for $\xi > 1$ large enough
\begin{align}
  \notag
  &\int_{\mathcal{P}}\PP_P\Big(\ell(\hat{\theta}_r(\mathbf{X}_n),\theta_r(P;\mathbf{X}_n)) > \varepsilon \Big) \mathrm{PY}_{\alpha,d,G}(\intd P)\\
  \notag
  &\qquad\qquad\qquad\geq \EE_{\Pi}\big[ \inf_{t \in \NNReals} \EE_{\Pi}\big(I(  \ell(t,\theta_r(P;\mathbf{X}_n)) > \varepsilon ) \mid \mathbf{X}_n\big)I\big(\alpha K_n \leq \min(n/2-1,\, \xi \EE_{\Pi}(\alpha K_n)) \big) \big]\\
  \label{eq:5ter}
  &\qquad\qquad\qquad\geq \Bigg(1 - \frac{\varepsilon/r}{1-\varepsilon}\sqrt{\frac{2(d+\xi \EE_{\Pi}(\alpha K_n))(d+n-1)}{n}} \Bigg)\PP_{\Pi}\Big(\alpha K_n \leq \min(n/2-1,\, \xi \EE_{\Pi}(\alpha K_n))  \Big)
\end{align}
Choosing the parameters of the PYP as $\alpha = 0$ and $d \in (0,1)$, we deduce that for every estimator $\hat\theta_r$, every $\varepsilon \in (0,1)$ and every $d \in (0,1)$
\begin{equation}
  \label{eq:4ter}
  \sup_{P\in \mathcal{P}}\PP_P\Big( \ell(\hat{\theta}_r,\theta_r) > \varepsilon \Big)%
  \geq 1 - \sqrt{2d} \frac{\varepsilon/r}{1-\varepsilon}.
\end{equation}
Since the previous is true for all $d \in (0,1)$, we deduce that for all estimator $\hat\theta_r$ and all $\varepsilon \in (0,1)$
\begin{equation*}
  \sup_{P\in \mathcal{P}}\PP_P\Big( \ell(\hat{\theta}_r,\theta_r) > \varepsilon \Big) = 1.
\end{equation*}

\begin{lemma}
  \label{lem:4}
  If $P \sim \mathrm{PY}_{\alpha,d,G}$ for $\alpha \in [0,1)$ and $d > -\alpha$, then the posterior distribution of $P$ based on observations $X_1,\dots,X_n \mid P \distiid P$ is the distribution of the random measure
  \begin{equation*}
    R_n\sum_{j=1}^{K_n}W_j\delta_{\tilde{X}_j} + (1 - R_n)Q_n
  \end{equation*}
  where $R_n \sim \mathrm{Beta}(n - \alpha K_n, d + \alpha K_n)$, $(W_1,\dots,W_{K_n}) \sim \mathrm{Dirichlet}(K_n;N_{1,n} - \alpha,\dots,N_{K_n,n}-\alpha)$, and $Q_n \sim \mathrm{PY}_{\alpha,d + \alpha K_n,G})$, all independently distributed. Here $\tilde{X}_1,\dots,\tilde{X}_{K_n}$ are the distinct values of $X_1,\dots,X_n$ and $N_{1,n},\dots,N_{K_n,n}$ their multiplicities.
\end{lemma}
\begin{proof}
  See \cite[Theorem~14.37]{Gho(17)}.
\end{proof}

  \begin{lemma}
    \label{lem:gammaf:ineq}
    For all $z > 0$
  \begin{equation*}
    z\log(z) - z + \frac{1}{2}\log\frac{2\pi}{z}%
    \leq \log\Gamma(z) \leq z\log(z) - z + \frac{1}{2}\log\frac{2\pi}{z}%
    + \frac{1}{12 z}.
  \end{equation*}
\end{lemma}
\begin{proof}
  See for instance \cite[Section~3.6]{Tem(96)}.
\end{proof}

\begin{lemma}
  \label{lem:beta-bayes}
  Let $X \sim \betaDist(a,b)$ with $a > 0$ and $b > 1$. Let $\delta_2> \delta_1 > 0$. Then,
  \begin{equation*}
    \sup_{t\in \NNReals}P\big(t\delta_1 \leq X \leq t\delta_2 \big) \leq \Bigg(\frac{\delta_2}{\delta_1} - 1 \Bigg)\sqrt{\frac{a(a + b-1)}{2\pi(b-1)}}e^{\frac{1}{12(a+b-1)}}.
  \end{equation*}
\end{lemma}%
\begin{proof}
  Let $g$ denote the density of the $\betaDist(a,b)$ distribution. Then
  \begin{align*}
    \sup_{t\geq 0 }P\big(t\delta_1 \leq X \leq t\delta_2 \big)
    &= \sup_{t\geq 0}\int_{t\delta_1}^{t\delta_2}g(x)\intd x\\
    &\leq \sup_{t\geq 0} t(\delta_2 - \delta_1)\sup_{x \in [t\delta_1,t\delta_2]}g(x)\\
    &= (\delta_2 - \delta_1) \sup_{u\in [\delta_1,\delta_2]} \sup_{t \geq 0} t g(t u).
  \end{align*}
  Note that $g(tu) = 0$ for $t > 1/u$ so that to find the supremum value of $t\mapsto tg(t u)$ it is enough to consider $t \in [0,1/u]$. But when $t \in [0,1/u]$ it is the case that
  \begin{align*}
    t g(t u)%
    &= \frac{\Gamma(a+b)}{\Gamma(a)\Gamma(b)} u^{a-1} t^{a}(1- u t)^{b-1}
  \end{align*}
  which because $b > 1$ and $a > 0$ attains a maximum at $t_{*} \coloneqq \frac{1}{u}\frac{a}{a + b -1}$ [notice that $t_{*} < 1/u$ as needed]. Therefore,
  \begin{align*}
    \sup_{t\geq 0}P\big(t\delta_1 \leq X \leq t\delta_2 \big)%
    &\leq (\delta_2 - \delta_1) \sup_{u\in [\delta_1,\delta_2]} \frac{\Gamma(a+b)}{\Gamma(a)\Gamma(b)} u^{a-1}\Big(\frac{1}{u}\frac{a}{a + b -1} \Big)^{a}\Big(1 - \frac{a}{a+b -1}\Big)^{b-1}\\
    &=\Big(\frac{\delta_2}{\delta_1} - 1\Big) \frac{\Gamma(a+b)}{(a+b-1)^{a+b-1}}\frac{a^{a}}{\Gamma(a)}\frac{(b-1)^{b-1}}{\Gamma(b)}.
  \end{align*}
  To further bound the rhs of the last display, we use Lemma~\ref{lem:gammaf:ineq} to deduce that
  \begin{align*}
    \Gamma(a+b)%
    &= (a+b - 1)\Gamma(a+b-1)\\
    &\leq (a+b-1) \cdot (a+b-1)^{a+b-1}e^{-(a+b-1)}\sqrt{\frac{2\pi}{a+b-1}}e^{\frac{1}{12(a+b-1)}},
  \end{align*}
  and,
  \begin{align*}
    \Gamma(a)%
    &\geq a^{a}e^{-a}\sqrt{\frac{2\pi}{a}},
  \end{align*}
  and,
  \begin{align*}
    \Gamma(b)%
    &= (b-1)\Gamma(b-1)%
      \geq (b-1)\cdot (b-1)^{b-1}e^{-(b-1)}\sqrt{\frac{2\pi}{b-1}}.
  \end{align*}
  Consequently,
  \begin{equation*}
    \sup_{t\geq 0}P\big(t\delta_1 \leq X \leq t\delta_2 \big)%
    \leq \Bigg(\frac{\delta_2}{\delta_1} - 1 \Bigg)\sqrt{\frac{a(a + b-1)}{2\pi(b-1)}}e^{\frac{1}{12(a+b-1)}}%
    \qedhere
  \end{equation*}
\end{proof}

  \subsection{Proof of Theorem~\ref{thm:impossibility-weak-consistency}}
  \label{sec:proof-theor-weakconsistency}
  We consider a Pitman-Yor process with parameters $(\alpha,d,G)$ as prior distribution, denoted in the sequel $\mathrm{PY}_{\alpha,d,G}$. We will choose $d > -\alpha$ accordingly at the end of the proof. The choice of $G$ is irrelevant for our purpose. Then, for any sequence $(\hat\theta_{r,n})_{n\geq 1}$ of $\mathbf{X}_n$-measurable estimators $\hat{\theta}_r$ and any $\varepsilon \in (0,1)$
  \begin{align*}
    \sup_{P\in \mathcal{P}}\limsup_{n\to\infty}\PP_P\Big( \ell(\hat{\theta}_{r,n},\theta_r(P;\mathbf{X}_n)) > \varepsilon \Big)
    &\geq \int_{\mathcal{P}} \limsup_{n\to\infty}\PP_P\Big( \ell(\hat{\theta}_{r,n},\theta_r(P;\mathbf{X}_n)) > \varepsilon \Big) \mathrm{PY}_{\alpha,d,G}(\intd P)\\
    &\geq \limsup_{n\to\infty} \int_{\mathcal{P}} \PP_P\Big( \ell(\hat{\theta}_{r,n},\theta_r(P;\mathbf{X}_n)) > \varepsilon \Big) \mathrm{PY}_{\alpha,d,G}(\intd P)
  \end{align*}
  where the second line follows from the (reverse) Fatou's lemma. Following carefully the computations from equation~\eqref{eq:5bis} to \eqref{eq:4ter}, it is seen that
  \begin{align*}
    \sup_{P\in \mathcal{P}}\limsup_{n\to\infty}\PP_P\Big( \ell(\hat{\theta}_{r,n},\theta_r(P;\mathbf{X}_n)) > \varepsilon \Big)%
    \geq 1 - 2\sqrt{d}\frac{\varepsilon/r}{1-\varepsilon}.
  \end{align*}
  Then the conclusion follows by taking $d\searrow 0$.

\subsection{Theorem~\ref{thm:3} : Alternative proof of a slightly weaker result}
\label{sec:proof-theor-refthm:3:alternative}

For $\omega_j \in (0,1)$ we let $P_j = (1-\omega_j)\delta_{\heartsuit} + \omega_j\delta_{\diamond}$, $j=1,2$. By Le Cam's two point method, for all $\varepsilon \in (0,1)$ and for all estimators $\hat{\theta}_r$%
\begin{equation*}
  \sup_{P\in \mathcal{P}}\PP_P\Big(%
  \ell(\hat{\theta}_r(\mathbf{X}_n),\theta_r(P;\mathbf{X}_n)) \geq \varepsilon
  \Big)
  \geq
    \frac{1}{2}\PP_{P_1}\Big(%
  \ell(\hat{\theta}_r(\mathbf{X}_n),\theta_r(P_1;\mathbf{X}_n)) \geq \varepsilon
    \Big)%
    + \frac{1}{2}\PP_{P_2}\Big(%
  \ell(\hat{\theta}_r(\mathbf{X}_n),\theta_r(P_2;\mathbf{X}_n)) \geq \varepsilon
  \Big).
\end{equation*}
Conditional on the event $E \coloneqq \Set{X_1 \ne \diamond,\dots,X_n\ne \diamond}$, it is the case that $\theta_r(P_j;\mathbf{X}_n) = \omega_j^r$ almost-surely under $P_j$, $j=1,2$. Therefore,
\begin{equation*}
  \PP_{P_j}\Big(%
  \ell(\hat{\theta}_r(\mathbf{X}_n),\theta_r(P_j;\mathbf{X}_n)) \geq \varepsilon%
  \bigcap E
  \Big)%
  = \PP_{P_j}\Big(%
  \Big|\frac{\hat{\theta}_r(\mathbf{X}_n)}{\omega_j^r} - 1 \Big| \geq \varepsilon
   \bigcap E\Big),\quad j=1,2.
\end{equation*}
Now we make the choice that
\begin{equation*}
  \omega_1 = \omega_2\Big(\frac{1-\varepsilon}{1 + \varepsilon} \Big)^{1/r}.
\end{equation*}
With this choice, observe that
\begin{align*}
  \Big|\frac{\hat{\theta}_r(\mathbf{X}_n)}{\omega_1^r} - 1 \Big|%
  < \varepsilon%
  &\iff \omega_1^r(1- \varepsilon) < \hat{\theta}_r(\mathbf{X}_n) < \omega_1^r(1+\varepsilon)\\
  &\iff \omega_2^r\frac{(1-\varepsilon)^2}{1+\varepsilon}< \hat{\theta}_r(\mathbf{X}_n) < \omega_2^r(1-\varepsilon)\\
  &\implies \Big|\frac{\hat{\theta}_r(\mathbf{X}_n)}{\omega_2^r} - 1\Big|%
    \geq \varepsilon.
\end{align*}
Letting $F \coloneqq \Set{ |\hat{\theta_r(\mathbf{X}_n)}/\omega_1^r - 1| \geq \varepsilon }$, it follows that
\begin{align*}
   \sup_{P\in \mathcal{P}}\PP_P\Big(%
  \ell(\hat{\theta}_r(\mathbf{X}_n),\theta_r(P;\mathbf{X}_n)) \geq \varepsilon
  \Big)
  &\geq \frac{1}{2}\PP_{P_1}(F \cap E)%
    + \frac{1}{2}\PP_{P_2}(F^c \cap E)\\
  &\geq \frac{1}{2}\PP_{P_1}(F \cap E)%
    + \frac{1}{2}\PP_{P_2}( (F \cap E)^c )%
    - \frac{1}{2}\PP_{P_2}(E^c)\\
  &\geq%
    \frac{1}{2}\Big(1 - \|\PP_{P_1} - \PP_{P_2}\|_{\textrm{TV}}\Big)%
    - \frac{n\omega_2}{2}
\end{align*}
where $\|\PP_{P_1} - \PP_{P_2}\|_{\textrm{TV}}$ denotes the total variation distance between $\PP_{P_1}$ and $\PP_{P_2}$. Note that $\|P_1 - P_2\|_{\textrm{TV}} = \omega_2 - \omega_1 \leq \omega_2$ which goes to zero as $\omega_2 \to 0$. This in particular implies that $\|\PP_{P_1} - \PP_{P_2}\|_{\textrm{TV}}\to 0$ as well when $\omega_2 \to 0$. Since the last display was true for all $\omega_2 \in (0,1)$, we deduce that for all $\varepsilon\in (0,1)$,
\begin{equation*}
  \sup_{P\in \mathcal{P}}\PP_P\Big(%
  \ell(\hat{\theta}_r(\mathbf{X}_n),\theta_r(P;\mathbf{X}_n)) \geq \varepsilon
  \Big) \geq \frac{1}{2}.
\end{equation*}
The previous display is a weaker statement than the one obtained using the first proof of Theorem~\ref{thm:3}, yet its proof is simpler and shed some lights on the impossibility of estimating the missing mass using a relative loss function.

\subsection{Proof of Proposition~\ref{pro:6}}
\label{sec:proof-pro:6}

The proof is nearly identical to the proofs in \cite{Ben(17)}. Nevertheless, we recall here the main steps as they will be useful later on to prove our Theorem~\ref{thm:5}.  Let $P\in \Sigma(\alpha,L)$ arbitrary and let $\Omega_n$ be the event on which:
  \begin{align*}
    \theta_r(P;\mathbf{X}_n) - \EE_P[\theta_r(P;\mathbf{X}_n)]%
    &> -\sqrt{2v_n(P)x}\\
    \theta_r(P;\mathbf{X}_n) - \EE_P[\theta_r(P;\mathbf{X}_n)] &< \sqrt{2v_n(P)x} + \left[\frac{2\max(\frac{c_0}{2},\, x + \log(n))}{n} \right]^r \frac{2x}{3}\\
    |M_{n,r} - \EE_P(M_{n,r})|
    &< \sqrt{8w_n(P)x} + \frac{2x}{3}
  \end{align*}
  with $w_n(P) \coloneqq \min( \max_{k\in\Set{r,r+1}}k\EE_P(M_{n,k}) ,\EE_P(C_{n,r}))$. By Propositions~\ref{pro:7} and~\ref{pro:2}, then event $\Omega_n$ occurs with probability at least $1 - 7e^{-x}$. Remark that when $P \in \Sigma(\alpha,L)$, then it must be that $\ell(\hat{\theta}_r^{\text{\tiny{(GT)}}},\theta_r) = \big|\frac{\hat{\theta}_r^{\text{\tiny{(GT)}}}(\mathbf{X}_n)}{\theta_r(P;\mathbf{X}_n)} - 1 \big|$ almost-surely. Then, it can be seen that on the event $\Omega_n$, provided $x \leq \frac{\EE_P(\theta_r(P;\mathbf{X}_n))^2}{8 v_n(P)}$ [which ensures that $\theta_r(P;\mathbf{X}_n) \geq \frac{1}{2}\EE_P(\theta_r(P;\mathbf{X}_n))$ on $\Omega_n$]
  \begin{align*}
    \ell(\hat{\theta}_r^{\text{\tiny{(GT)}}},\theta_r)%
    &= \frac{|\hat{\theta}_r^{\text{\tiny{(GT)}}}(\mathbf{X}_n) - \theta_r(P;\mathbf{X}_n)|}{\theta_r(P;\mathbf{X}_n)}\\
    &\leq%
      2\frac{\frac{1}{\binom{n}{r}} |M_{n,r} - \EE_P(M_{n,r})| + |\theta_r(P;\mathbf{X}_n) - \EE_P(\theta_r(P;\mathbf{X}_n))|}{\EE_P(\theta_r(P;\mathbf{X}_n))}\\
    &\leq%
      \frac{2[\sqrt{8w_n(P)x} + \frac{2x}{3}]}{\binom{n}{r}\EE_P(\theta_r(P;\mathbf{X}_n))}%
      + \frac{2[\sqrt{2v_n(P)x} + (\frac{2\max(\frac{c_0}{2}, x + \log(n)}{n})^r\frac{2x}{3}]}{\EE_P(\theta_r(P;\mathbf{X}_n))}.
  \end{align*}
  Furthermore $\EE_P(\theta_r(P;\mathbf{X}_n)) = \EE_P(M_{n+r,r})/\binom{n+r}{r}$, and it is proven in equation~\ref{eq:42} that $v_n(P) \leq \EE_P(C_{n,r})/\binom{n}{2r} $. Thus,
  \begin{align*}
    \ell(\hat{\theta}_r^{\text{\tiny{(GT)}}},\theta_r)%
    &\leq \Bigg(\frac{2}{\binom{n}{r}} + \frac{1}{\sqrt{\binom{n}{2r}}} \Bigg) \frac{\binom{n+r}{r}\sqrt{8\EE_P(C_{n,r})x}}{\EE_P(M_{n+r,r})}%
      + \Bigg(\frac{1}{\binom{n}{r}}%
      + \frac{2^r\max(\frac{c_0}{2},\,x+\log(n))^r}{n^r}%
      \Bigg)\frac{\binom{n+r}{r}4x}{3\EE_P(M_{n+r,r})}
  \end{align*}
  Using the inequality $\frac{n^k}{k^k} \leq \binom{n}{k} \leq \frac{n^ke^k}{k^k}$, we deduce that for all $x \leq \frac{\EE_P(\theta_r(P;\mathbf{X}_n))^2}{8 v_n(P)} = \frac{\EE_P(M_{n+r,r})^2}{8\EE_P(M_{n+2r,2r})}\frac{\binom{n+2r}{2r}}{\binom{n+r}{r}^2}$ with probability at least $1-7e^{-x}$
  \begin{equation}
    \label{eq:59jul}
    \ell(\hat{\theta}_r^{\text{\tiny{(GT)}}},\theta_r)%
    \leq%
    2(4e)^r\frac{\sqrt{8\EE_P(C_{n,r})x}}{\EE_P(M_{n+r,r})}%
    + (2e)^r\frac{4\max(1,\,x+\log(n))^rx}{\EE_P(M_{n+r,r})}.
  \end{equation}
      Finally, it is folklore results (see for instance Theorem~4.2 in \cite{Ben(17)}, or in \cite{Kar(67)}, or in \cite{Gne(07)}, or in Appendix~E in \cite{Fav(22)}) that when $P\in \Sigma(\alpha,L)$ the following asymptotic equivalence hold as $n\to \infty$:
  \begin{align*}
    \EE_P(M_{n+r,r}) &\sim \frac{\alpha \Gamma(r - \alpha)}{r!}\cdot n^{\alpha}L(n) \\
    \EE_P(M_{n+2r,2r}) &\sim \frac{\alpha \Gamma(2r - \alpha)}{(2r)!}\cdot n^{\alpha}L(n)\\
    \EE_P(C_{n,r}) &\sim \frac{\Gamma(r - \alpha)}{(r-1)!}\cdot n^{\alpha}L(n) \\
    \frac{\binom{n+2r}{2r}}{\binom{n+r}{r}^2}%
    &\sim \frac{(r!)^2}{(2r)!}
  \end{align*}
 Hence the result.

\subsection{Proof of Theorem~\ref{thm:4}}
\label{sec:proof-thm:4}

Using a density argument, we establish that the minimax risk over $\Sigma(\alpha,L)$ is the same as the risk over $\mathcal{P}$, then the result follows from the Theorem~\ref{thm:3}. In particular, in Proposition~\ref{pro:1} below, we prove that $\Sigma(\alpha,L)$ is dense in $\mathcal{P}$ for the topology induced by the distance $(P,Q) \mapsto \|P^{\otimes n} - Q^{\otimes n}\|_{\mathrm{TV}}$ [it is a well-known fact that this topology is equivalent to the topology induced by the distance $(P,Q) \mapsto \|P - Q\|_{\mathrm{TV}}$]. Consequently, for all estimator $\hat{\theta}_r$ and all $\delta > 0$, and for all $P\in \mathcal{P}$ and all $\varepsilon \in (0,1)$, we can find $Q \in \Sigma(\alpha,L)$ such that
\begin{align*}
  \PP_P\big( \ell(\hat{\theta}_r,\theta_r) \geq \varepsilon \big)%
  &\leq \PP_Q\big( \ell(\hat{\theta}_r,\theta_r) \geq \varepsilon \big)%
    + \|P^{\otimes n} - Q^{\otimes n}\|_{\mathrm{TV}}\\
  &\leq \sup_{Q\in \Sigma(\alpha,L)}\PP_Q\big( \ell(\hat{\theta}_r,\theta_r) \geq \varepsilon \big)%
    + \delta.
\end{align*}
Hence for all estimator $\hat{\theta}_r$ and all $\varepsilon \in (0,1)$
\begin{align}
  \label{eq:1}
  \sup_{P\in \mathcal{P}}\PP_P\big( \ell(\hat{\theta}_r,\theta_r) \geq \varepsilon \big)%
  &\leq \sup_{Q\in \Sigma(\alpha,L)}\PP_Q\big( \ell(\hat{\theta}_r,\theta_r) \geq \varepsilon \big).
\end{align}
On the other hand, it is obvious that for all estimator $\hat{\theta}_r$ and all $\varepsilon \in (0,1)$
\begin{align}
  \label{eq:3}
  \sup_{P\in \mathcal{P}}\PP_P\big( \ell(\hat{\theta}_r,\theta_r) \geq \varepsilon \big)%
  &\geq \sup_{Q\in \Sigma(\alpha,L)}\PP_Q\big( \ell(\hat{\theta}_r,\theta_r) \geq \varepsilon \big).
\end{align}
Combining equations~\eqref{eq:1} and~\eqref{eq:3}, we obtain that for all $\varepsilon \in (0,1)$
\begin{align*}
  \inf_{\hat{\theta}_r}\sup_{P\in \mathcal{P}} \PP_P\big( \ell(\hat{\theta}_r,\theta_r) \geq \varepsilon \big)%
  &= \inf_{\hat{\theta}_r}\sup_{Q\in \Sigma(\alpha,L)} \PP_Q\big( \ell(\hat{\theta}_r,\theta_r) \geq \varepsilon \big).
\end{align*}

\subsection{Proof of Proposition~\ref{pro:1}}
\label{sec:proof-prop-refpr-1}

Let us write $\mathfrak{P}$ the support of $P$, and $P = \sum_{s \in \mathfrak{P}}P_s\delta_s$. Without loss of generality we can assume that $\mathfrak{P}$ is finite and $\min_sP_s \geq \delta$ for some $\delta \in (0,1)$. This is because the set of measures with finite support and with masses bounded from below is dense in $\mathcal{P}$ for the total variation metric. Now we let $M = \sum_{t\in \mathfrak{M}}M_t\delta_t$ be a probability measure on the set of symbols $\mathfrak{M}$ such that $\mathfrak{M} \cap \mathfrak{P} = \varnothing$ and $\lim_{x\to 0}x^{\alpha}\bar{F}_M(x) = L \eta^{-\alpha}$ for $\eta>0$ small. Such a measure $M$ always exists by the Lemma~\ref{lem:2} below. Now build the measure $Q = (1-\eta)P + \eta M$. Whenever $x < (1-\eta)\delta$, it is the case that
\begin{equation*}
  \bar{F}_Q(x)%
  =
  |\mathfrak{P}| +  \sum_{t\in \mathfrak{M}}I(\eta M_t > x).
\end{equation*}
Since $|\mathfrak{P}| < \infty$ and $\alpha \in (0,1)$, deduce that $\lim_{x\to 0}x^{\alpha}\bar{F}_Q(x) = \lim_{x \to 0}x^{\alpha}\bar{F}_M(x/\eta) = L$ by construction. Hence $Q\in \Sigma(\alpha,L)$ and
\begin{equation*}
  \|P - Q\|_{\mathrm{TV}}%
  = \eta.
\end{equation*}
Since $\eta$ was taken arbitrarily small, this concludes the proof.

\begin{lemma}
  \label{lem:2}
  For all $L > 0$ and all $\alpha \in (0,1)$ there exists a probability
  distribution $P$ on $\Nats$ such that $\lim_{x\to 0}x^{\alpha}\bar{F}_P(x) = L$.
\end{lemma}
\begin{proof}
  We will prove the lemma by constructing such a distribution. We take $c > 0$
  arbitrary for now and $b \geq 1$ integer also arbitrary for now; their value
  will be chosen accordingly at the end of the day. We consider any
  $Q \equiv (Q_1,Q_2,\dots)$ [we use the convention $\PP_Q(X = i) = Q_i$] such that
  $|\Set{i \in \Nats \given Q_i = c 2^{-k/\alpha}}| = b 2^{k}$ for all
  $k\geq 0$ integer. We
  pick the value of $c > 0$ (as a function of $b$) such that $Q$ is a proper
  probability distribution. We see immediately that it must be that [remark that
  $b 2^k$ is integer]
  \begin{align*}
    1
    &= \sum_{i\geq 1}Q_i
    = \sum_{k\geq 0} c 2^{-k/\alpha}|\Set{i \in \Nats \given Q_i = c 2^{-k/\alpha}}|
    = cb \sum_{k\geq 0}2^{-k(1/\alpha - 1)}
    = cb \cdot \frac{2^{1/\alpha}}{2^{1/\alpha} - 2}.
  \end{align*}
  This establishes that we must have
  \begin{align*}
    c = \frac{2^{1/\alpha}-2}{b 2^{1/\alpha}}.
  \end{align*}
  On the other hand we have,
  \begin{align*}
    \bar{F}_Q(x)%
    &= \sum_{i\geq 1}I(Q_i > x)\\
    &= \sum_{k\geq 0}|\Set{i \in \Nats \given Q_i = c 2^{-k/\alpha}}|I(c 2^{-k/\alpha} > x)\\
    &= b \sum_{k\ge 0}2^kI(2^k < (x/c)^{-\alpha})\\
    &= b(2^{k_{*}+1} - 1)
  \end{align*}
  where $k_{*}$ is the only integer satisfying
  $(x/c)^{-\alpha} \leq 2^{k_{*}} \leq (x/c)^{-\alpha} + 1$. It follows
  immediately that
  \begin{align*}
    \lim_{x\to 0}x^{\alpha}\bar{F}_Q(x)%
    &= 2bc^{\alpha}%
      = b^{1-\alpha}\big(2^{1/\alpha} - 2 \big)^{\alpha}
  \end{align*}
  Therefore if we choose [recall that $b$ must be an integer!]
  \begin{align*}
    b = \Bigg\lceil \Big(\frac{L}{(2^{1/\alpha} - 2)^{\alpha}} \Big)^{1/(1-\alpha)} \Bigg\rceil.
  \end{align*}
  we have that $\lim_{x\to 0}x^{\alpha}\bar{F}_Q(x) = L'$ for a $L' \geq L$. We can then obtain the desired $P$ by taking $P = (1 - \omega)Q + \omega \delta_1$ for an appropriate choice of $\omega \in [0,1]$. Indeed for such $P$
  \begin{align*}
    \lim_{x \to 0}x^{\alpha}\bar{F}_P(x)
    &= \lim_{x \to 0}x^{\alpha}\Bigg( \sum_{j\geq 1}I((1-\omega)Q_j > x) + I(\omega > x) \Bigg)\\
    &= \lim_{x\to 0}x^{\alpha} \bar{F}_Q\Big(\frac{x}{1-\omega}\Big)\\
    &= (1-\omega)^{\alpha}L'.
  \end{align*}
  So by taking $1 - \omega = (L/L')^{1/\alpha}$ (which is between $0$ and $1$) we are done.
\end{proof}

\subsection{Proof of Theorem~\ref{thm:5}}
\label{sec:proof-thm:5}

\subsubsection{Proof of the upper bound}
\label{sec:proof-upper-bound-1}

The starting point of the proof is the equation~\eqref{eq:59jul}. The only thing we need to conclude and that differs from the Proposition~\ref{pro:6} is a uniform control over $\EE_P(M_{n+r,r})$ and $\EE_P(C_{n,r})$ when $P \in \Sigma(\alpha,\beta,C,C')$. Such uniform control are obtained in the Lemmas~\ref{lem:8} and~\ref{lem:38jul} below.

In particular, using the estimates that $\binom{n+2r}{2r} \geq \frac{(n+2r)^{2r}}{(2r)^{2r}}$ and $\binom{n+r}{r} \leq \frac{(n+r)^re^r}{r^r}$ together with the estimates in Lemma~\ref{lem:38jul}, it is found that when $n \geq \max\{1,\, [\frac{12C'\Gamma(2r-\alpha+\beta+1)}{\Gamma(2r-\alpha)}]^{1/\beta}\}$ [see that $k\mapsto \frac{\Gamma(k-\alpha+\beta+1)}{\Gamma(k-\alpha)}$ is monotone increasing]
  \begin{align*}
   \frac{\EE_P(M_{n+r,r})^2\binom{n+2r}{2r} }{\EE_P(M_{n+2r,2r})\binom{n+r}{r}^2}
    &\geq%
      \frac{\frac{1}{4}L_{\alpha}(P)^2\frac{\Gamma(r-\alpha)^2}{(r!)^2}(n+r)^{2\alpha} \frac{(n+2r)^{2r}}{(2r)^{2r}} }{3L_{\alpha}(P)\frac{\Gamma(2r-\alpha)}{(2r)!}(n+2r)^{\alpha} \frac{(n+r)^{2r}e^{2r}}{r^{2r}}}\\
    &= \frac{L_{\alpha}(P)}{12}\frac{\Gamma(r-\alpha)^2}{(r!)^2}\frac{(2r)!}{\Gamma(2r-\alpha)}\frac{(1 + \frac{r}{n})^{\alpha}}{(1+\frac{2r}{n})^{\alpha}}\Big(\frac{e}{2} \Big)^{2r}n^{\alpha}\\
    &\geq \frac{C}{12}r^{-1-\alpha}\Big(\frac{e}{2} \Big)^{2r} n^{\alpha}
  \end{align*}
  where the last line follows because $\frac{\Gamma(r-\alpha)}{r!} = \Gamma(r-1 + (1-\alpha)){r\Gamma(r-1+1)} \geq r^{-\alpha-1}$ by Gautschi's inequality, and $\frac{(2r)!}{\Gamma(2r-\alpha)} = \frac{2r\Gamma(2r-1 + 1)}{\Gamma(2r-1 + (1-\alpha))}\geq 2r(2r-1)^{\alpha}\geq 2r^{\alpha+1}$ also by Gautschi's inequality. Hence, whenever $x \leq \frac{C(e/2)^{2r}}{96r^{1+\alpha}}n^{\alpha}$, equation~\eqref{eq:59jul} guarantees that with probability at least $1-7e^{-x}$
  \begin{align*}
    \ell(\hat{\theta}_r^{\text{\tiny{(GT)}}},\theta_r)%
    &\leq%
    2(4e)^r\frac{\sqrt{8\EE_P(C_{n,r})x}}{\EE_P(M_{n+r,r})}%
    + (2e)^r\frac{4\max(1,\,x+\log(n))^rx}{\EE_P(M_{n+r,r})}\\
    &\leq 16r(4e)^r\sqrt{\frac{1 + \frac{3C'\Gamma(r-\alpha+\beta)}{\Gamma(r-\alpha)n^{\beta}}}
      {\frac{\Gamma(r-\alpha)}{\Gamma(r)}L_{\alpha}(P)n^{\alpha}}}\sqrt{x}%
      + \frac{8r(2e)^r(x + \log(n))^r}{\frac{\Gamma(r-\alpha)}{\Gamma(r)} \sqrt{L_{\alpha}(P)n^{\alpha}}\cdot\sqrt{C}n^{\alpha/2}}x
  \end{align*}
  where the last line holds by Lemmas~\ref{lem:8} and~\ref{lem:38jul} under the assumption that $n \geq \max\{3,\, [\frac{12C'\Gamma(2r-\alpha+\beta+1)}{\Gamma(2r-\alpha)}]^{1/\beta}\}$.

\begin{lemma}
  \label{lem:8}
  Let $C_{n,k} \coloneqq \sum_{j\geq k}M_{n,j}$. For all $\alpha \in (0,1)$, all $\beta > 0$, all $C > 0$, all $C'>0$, all $P\in \Sigma(\alpha,\beta,C,C')$, all $n \geq 1$ and all $0 \leq k \leq n$
  \begin{equation*}
    \EE_P(C_{n,k})%
      \leq 2n^{\alpha}L_{\alpha}(P)\frac{\Gamma(k-\alpha)}{\Gamma(k)}\Bigg(1%
        + \frac{3C'\Gamma(k-\alpha+\beta)}{\Gamma(k-\alpha)n^{\beta}}\Bigg).
      \end{equation*}
\end{lemma}
\begin{proof}
  The proof is similar to Lemma~E.3 in \cite{Fav(22)}, we only sketch the main arguments. In the next we assume $P\in \Sigma(\alpha,\beta,C,C')$ is arbitrary. It can be established that (see Lemma~E.2 in \cite{Fav(22)}) that
  \begin{equation*}
    \EE_P(C_{n,k})%
    = k\binom{n}{k}\int_0^1\bar{F}_P(x)x^{k-1}(1-x)^{n-k}\intd x
  \end{equation*}
  and by straightforward computations that
  \begin{equation*}
    \int_0^1 x^{-\alpha}\cdot x^{k-1}(1-x)^{n-k}\intd x%
    = \frac{(n-k)!\Gamma(k-\alpha)}{\Gamma(n-\alpha+1)}.
  \end{equation*}
  Therefore,
  \begin{align}
    \notag
    \Big|\EE_P(C_{n,k}) - L_{\alpha}(P)\frac{\Gamma(k-\alpha)}{\Gamma(k)}\frac{\Gamma(n+1)}{\Gamma(n-\alpha + 1)} \Big|
    &\leq k\binom{n}{k}\int_0^1|\bar{F}_P(x) - L_{\alpha}(P)x^{-\alpha}|x^{k-1}(1-x)^{n-k}\intd x\\
    \notag
    &\leq C'L_{\alpha}(P) k\binom{n}{k}\int_0^1x^{k +\beta - \alpha -1}(1-x)^{n-k}\intd x\\
    \label{eq:48jul}
    &= C' L_{\alpha}(P)\frac{\Gamma(k-\alpha+\beta)}{\Gamma(k)}\frac{\Gamma(n+1)}{\Gamma(n-\alpha+\beta+1)}.
  \end{align}
    Hence,
    \begin{align*}
      \EE_P(C_{n,k})%
      &\leq L_{\alpha}(P)\frac{\Gamma(k-\alpha)}{\Gamma(k)}\frac{\Gamma(n+1)}{\Gamma(n+1-\alpha)}\Bigg(1%
        + \frac{C'\Gamma(k-\alpha+\beta)}{\Gamma(k-\alpha)}\frac{\Gamma(n-\alpha+1)}{\Gamma(n-\alpha+1+\beta)}\Bigg)
    \end{align*}
    Next we use that $\frac{\Gamma(n+1)}{\Gamma(n+1-\alpha)} \leq (n+1)^{\alpha} \leq 2n^{\alpha}$ by Gautschi's inequality. Also by Gautschi's inequality if $0 < \beta < 1$ then $\frac{\Gamma(n-\alpha+1)}{\Gamma(n-\alpha+1+\beta)} = \frac{1}{n-\alpha+1}\frac{\Gamma(n-\alpha+2)}{\Gamma(n-\alpha+1+\beta)} \leq \frac{(n-\alpha+2)^{1-\beta}}{n-\alpha+1} = 3n^{-\beta}$. If $\beta \geq 1$, then we remark that $\Gamma(n-\alpha+1 + \beta) = \Gamma(n-\alpha+1)\EE(X^{\beta})$ where $X\sim \gammaDist(n-\alpha+1,1)$; so that $\Gamma(n-\alpha+1 + \beta) \geq \Gamma(n-\alpha+1)(n-\alpha+1)^{\beta}\geq \Gamma(n-\alpha+1)n^{\beta}$ by Jensen's inequality. At the end of the day
    \begin{equation*}
      \EE_P(C_{n,k})%
      \leq 2n^{\alpha}L_{\alpha}(P)\frac{\Gamma(k-\alpha)}{\Gamma(k)}\Bigg(1%
        + \frac{3C'\Gamma(k-\alpha+\beta)}{\Gamma(k-\alpha)n^{\beta}}\Bigg). \qedhere
      \end{equation*}
\end{proof}

  \begin{lemma}
    \label{lem:38jul}
    For all $\alpha \in (0,1)$, all $\beta > 0$, all $C > 0$, all $C'>0$, all $P\in \Sigma(\alpha,\beta,C,C')$, all $k\geq 1$, and all $n \geq \max\{1,\, [\frac{12C'\Gamma(k-\alpha+\beta+1)}{\Gamma(k-\alpha)}]^{1/\beta}\}$,
    \begin{equation*}
      \frac{1}{2}L_{\alpha}(P)\frac{\Gamma(k-\alpha)}{k!}(n+k)^{\alpha}%
      \leq \EE_P(M_{n+k,k})%
        \leq
      3 L_{\alpha}(P)\frac{\Gamma(k-\alpha)}{k!}(n+k)^{\alpha}.%
    \end{equation*}
  \end{lemma}
  \begin{proof}
    In the whole proof we take $P \in \Sigma(\alpha,\beta,C,C')$ arbitrary. Since $M_{n,k} = C_{n,k} - C_{n,k+1}$ we deduce from equation~\eqref{eq:48jul} in the proof of Lemma~\ref{lem:8} that
    \begin{align*}
      &\Big|\EE_P(M_{n+k,k}) - L_{\alpha}(P)\frac{\Gamma(k-\alpha)}{k!}\frac{\Gamma(n+k+1)}{\Gamma(n+k-\alpha+1)} \Big|\\%
      &\qquad\qquad\leq C'L_{\alpha}(P)\Bigg(\frac{\Gamma(k-\alpha+\beta)}{\Gamma(k)} + \frac{\Gamma(k-\alpha+\beta+1)}{\Gamma(k+1)} \Bigg)\frac{\Gamma(n+k+1)}{\Gamma(n+k-\alpha+\beta+1)}\\
      &\qquad\qquad= C'L_{\alpha}(P)\frac{2k-\alpha+\beta}{k}\frac{\Gamma(k-\alpha+\beta)}{\Gamma(k)}\frac{\Gamma(n+k+1)}{\Gamma(n+k-\alpha+1+\beta)}
    \end{align*}
    Hence,
    \begin{align*}
      \Bigg|\frac{\EE_P(M_{n+k,k})}{L_{\alpha}(P)\frac{\Gamma(k-\alpha)}{k!}\frac{\Gamma(n+k+1)}{\Gamma(n+k-\alpha+1)} }%
      -1 \Bigg|
      &\leq C'(2k-\alpha+\beta) \frac{\Gamma(k-\alpha+\beta)}{\Gamma(k-\alpha)}%
        \frac{\Gamma(n+k-\alpha+1)}{\Gamma(n+k-\alpha+1+\beta)}\\
      &\leq 6C'\frac{\Gamma(k-\alpha+\beta+1)}{\Gamma(k-\alpha)}(n+k)^{-\beta}\\
      &\leq \frac{1}{2}
    \end{align*}
    where the second line follows from the same reasoning as the lines after equation~\eqref{eq:48jul} in the proof of Lemma~\ref{lem:8}, and the last line because of the assumptions in the statement of the lemma. Hence, under the assumptions of the lemma
    \begin{equation*}
      \frac{1}{2}L_{\alpha}(P)\frac{\Gamma(k-\alpha)}{k!}\frac{\Gamma(n+k+1)}{\Gamma(n+k-\alpha+1)}%
      \leq \EE_P(M_{n+k,k})%
        \leq
      \frac{3}{2}L_{\alpha}(P)\frac{\Gamma(k-\alpha)}{k!}\frac{\Gamma(n+k+1)}{\Gamma(n+k-\alpha+1)}.%
    \end{equation*}
    Finally, the conclusion follows from Gautschi's inequality which implies that
    \begin{equation*}
      (n+k)^{\alpha} \leq \frac{\Gamma(n+k+1)}{\Gamma(n+k+1-\alpha)}%
      \leq
      (n+k+1)^{\alpha} \leq 2(n+k)^{\alpha}.
      \qedhere
    \end{equation*}
  \end{proof}

\subsubsection{Proof of the lower bound}
\label{sec:proof-thm:6}

We use the traditional approach that the minimax risk is always larger than the Bayes risk relative to any choice of prior. We consider a Pitman-Yor process \cite[Section~14.4]{Gho(17)} with parameters $(\alpha,d,G)$ as prior distribution, denoted in the sequel $\mathrm{PY}_{\alpha,d,G}$. We choose $d = 0$ [which is allowed here since $\alpha > 0$] and $G$ is any nonatomic probability measure. Then, for any estimator $\hat{\theta}_r$, for a constant $A>0$ to be chosen accordingly,
\begin{align}
  \notag
  &\sup_{P\in \Sigma(\alpha,\beta,C,C')}\PP_P\Big( \sqrt{n^{\alpha}L_{\alpha}(P)} \ell(\hat{\theta}_r,\theta_r) > AC^{1/2} \Big)\\%
  \notag
  &\qquad\qquad\geq \int_{\Sigma(\alpha,\beta,C,C')}\PP_P\Big( \sqrt{n^{\alpha}L_{\alpha}(P)} \ell(\hat{\theta}_r,\theta_r) > AC^{1/2} \Big) \mathrm{PY}_{\alpha,0,G}(\intd P)\\
  &\qquad\qquad\geq \int_{\mathcal{P}}\PP_P\Big(n^{\alpha/2}\ell(\hat{\theta}_r,\theta_r) > A \Big) \mathrm{PY}_{\alpha,0,G}(\intd P)%
    - \mathrm{PY}_{\alpha,0,G}\big( \Sigma(\alpha,\beta,C,C')^c \big)
        \label{eq:6}
\end{align}
The rest of this section is dedicated to bounding each of the term involved in the last rhs and showing that their sum is at least $1/4$ by carefully tuning $A$ and $C$.

\paragraph{Control of the first term of \eqref{eq:6}}

The first term is bounded exactly as in Section~\ref{sec:proof-theor-refthm:3} by setting $\varepsilon = An^{-\alpha/2}$ all along the proof [taking the care that $An^{-\alpha/2} < 1$]. In particular, the equation~\eqref{eq:5ter} remains true here, so that we have
  \begin{align*}
    \int_{\mathcal{P}}\PP_P\Big(n^{\alpha/2}\ell(\hat{\theta}_r,\theta_r) > A \Big) \mathrm{PY}_{\alpha,0,G}(\intd P)%
    &\geq \Bigg(1 - \frac{1}{r}\frac{An^{-\alpha/2}\sqrt{2\xi \EE_{\Pi}(\alpha K_n)}}{1 - An^{-\alpha/2}} \Bigg)\PP_{\Pi}(\alpha K_n\leq \min(n/2-1,\xi \EE_{\Pi}(\alpha K_n))\\
    &\geq \Bigg(1 - \frac{1}{r}\frac{An^{-\alpha/2}\sqrt{2\xi \EE_{\Pi}(\alpha K_n)}}{1 - An^{-\alpha/2}} \Bigg)\Bigg(1 - \frac{\EE_{\Pi}(\alpha K_n)}{\min(n/2-1,\xi \EE_{\Pi}(\alpha K_n))}\Bigg)\\
    &= \Bigg(1 - \frac{1}{r}\frac{An^{-\alpha/2}\sqrt{2\xi \EE_{\Pi}(\alpha K_n)}}{1 - An^{-\alpha/2}} \Bigg)\Bigg(1 - \max\Big(\frac{1}{\xi},\, \frac{\EE(\alpha K_n)}{n/2-1}\Big)\Bigg)\\
    &\geq \Bigg(1 - \frac{1}{r}\frac{An^{-\alpha/2}\sqrt{2\xi \EE_{\Pi}(\alpha K_n)}}{1 - An^{-\alpha/2}} \Bigg)\Bigg(1 - \max\Big(\frac{1}{\xi},\, \frac{4\EE(\alpha K_n)}{n}\Big)\Bigg)
  \end{align*}
  where the last line follows by the assumption that $n\geq 4$. It is well-known (see for instance \cite[Exercise~3.2.9]{Pit(06)}) that $\EE_{\Pi}(\alpha K_n) = \frac{\Gamma(n+\alpha)}{\Gamma(\alpha)\Gamma(n)} \leq \alpha n^{\alpha}$, where the last estimate follows from the fact that $\Gamma(\alpha)\leq 1/\alpha$ for $\alpha \in (0,1)$ and from Gautschi's inequality. Then whenever $A \leq \min(\frac{r}{8\sqrt{6\alpha}},\, \frac{n^{\alpha/2}}{2})$ and $n \geq (12\alpha)^{1/(1-\alpha)}$:
  \begin{align}
    \label{eq:50:final}
    \int_{\mathcal{P}}\PP_P\Big(n^{\alpha/2}\ell(\hat{\theta}_r,\theta_r) > A \Big) \mathrm{PY}_{\alpha,0,G}(\intd P)%
    \geq \frac{1}{2}.
  \end{align}

\paragraph{Control of the second term of \eqref{eq:6}}

Here we bound the term $\mathrm{PY}_{\alpha,0,G}(\Sigma(\alpha,\beta,C,C')^c)$. In particular, we show in the Lemma~\ref{lem:1a} and~\ref{lem:1b} that $\mathrm{PY}_{\alpha,0,G}(\Sigma(\alpha,\beta,C,C')^c) \leq \frac{1}{4}$ whenever $C \leq \frac{1}{64\pi\Gamma(1-\alpha)}$ and $C' > C(\alpha,\beta)$ for a constant $C(\alpha,\beta)$ depending solely on $(\alpha,\beta)$.

To prove the lemmas, we recall the following helpful construction of the Pitman-Yor process from \cite[Example~14.47]{Gho(17)}. Let $(\Omega,\mathcal{A},\PP_{\alpha})$ be a probability space on which is defined a Poisson process with with mean intensity measure $\rho_{\alpha}(\intd x) = \frac{\alpha x^{-\alpha-1}\intd x}{\Gamma(1-\alpha)}$ [the \textit{so-called} stable process] and a collection $(S_i)_{i\geq 1}$ of iid random variables with marginal distribution $G$. Also let $J_1 \geq J_2 \geq \dots$
the ordered jumps of the Poisson process and $T \coloneqq \sum_{i\geq 1}J_i$. Note that $T$ has the density  $f_{\alpha}$ of the stable distribution, \textit{i.e.} $\int_0^{\infty}e^{-\lambda t} f_{\alpha}(t)\intd t = e^{-\lambda^{\alpha}}$ for all $\lambda \geq 0$. Furthermore, $\mathrm{PY}_{\alpha,0,G}$ is the distribution of the random probability measure $P = \sum_{j\geq 1}\frac{J_i}{T}\delta_{S_j}$.

\begin{lemma}
  \label{lem:1a}
  For all $\alpha \in (0,1)$, $\mathrm{PY}_{\alpha,0,G}(L_{\alpha}(P) < C )  \leq \sqrt{\pi\Gamma(1-\alpha)C}$.
\end{lemma}
\begin{proof}
  Using the construction of the Pitman-Yor process described above, we see that
  \begin{align*}
    \mathrm{PY}_{\alpha,0,G}(L_{\alpha}(P)< C )%
    &= \PP_{\alpha}\Big(\lim_{x\to 0}x^{\alpha}\bar{F}_P(x) < C \Big)\\
    &= \PP_{\alpha}\Big(\lim_{x\to 0}x^{\alpha}\sum_{i\geq 1}I(J_i > T x)  < C\Big)\\
    &= \PP_{\alpha}\Big(T^{-\alpha}\lim_{y\to 0}y^{\alpha}\sum_{i\geq 1}I(J_i > y) < C\Big).%
  \end{align*}
  Now it is a well-known fact that $\lim_{y\to 0}y^{\alpha} \sum_{j\geq 1}I(J_j > y) = \frac{1}{\Gamma(1-\alpha)}$ $\PP_{\alpha}$-almost-surely. Consequently,
  \begin{align*}
    \mathrm{PY}_{\alpha,0,G}(L_{\alpha}(P) < C)%
    &= \PP_{\alpha}\big(T > [C\Gamma(1-\alpha)]^{-1/\alpha} \big)\\
    &\leq \sqrt{C\Gamma(1-\alpha)}\EE_{\alpha}(T^{\alpha/2})\\
    &= \sqrt{C\Gamma(1-\alpha)}\frac{\Gamma(1/2)}{\Gamma(1-\alpha/2)}\\
    &\leq \sqrt{\pi\Gamma(1-\alpha)C}.
  \end{align*}%
\end{proof}

\begin{lemma}
  \label{lem:1b}
  For all $\alpha \in (0,1)$, for all $0 < \beta < \frac{\alpha}{2}$, and for all $\varepsilon > 0$ there exists $C' > 0$ such that
  \begin{equation*}
    \mathrm{PY}_{\alpha,0,G}\Big(\sup_{x\in (0,1)}x^{-\beta}\Big| \frac{\bar{F}_P(x)}{L_{\alpha}(P)x^{-\alpha}} -1\Big| > C' \Big) \leq \varepsilon.
  \end{equation*}
\end{lemma}
\begin{proof}
  Recall that $\mathrm{PYP}_{\alpha,0,G}$ is the law of the random measure $P = \sum_{i\geq 1}\frac{J_i}{T}\delta_{S_i}$.
  With the same arguments as in the proof of Lemma~\ref{lem:1a}, we have that $\bar{F}_P(x) = \sum_{i\geq 1}I(J_i > Tx)$ and $L_{\alpha}(P) = \frac{T^{-\alpha}}{\Gamma(1-\alpha)}$ almost-surely under $\PP_{\alpha}$. We deduce that under $\PP_{\alpha}$
  \begin{align*}
    \sup_{x\in (0,1)}x^{-\beta}\Big| \frac{\bar{F}_P(x)}{L_{\alpha}(P)x^{-\alpha}} -1\Big|
    &\overset{d}{=} T^{\beta} \Gamma(1-\alpha)\sup_{x\in (0,1)} (Tx)^{-\beta + \alpha}\Big|\sum_{i\geq 1}I(J_i > Tx) - \frac{(Tx)^{-\alpha}}{\Gamma(1-\alpha)}  \Big|\\
    &= T^{\beta} \Gamma(1-\alpha)\sup_{x\in (0,T)} x^{-\beta + \alpha}\Big|\sum_{i\geq 1}I(J_i > x) - \frac{x^{-\alpha}}{\Gamma(1-\alpha)}  \Big|.
  \end{align*}
  By a famous result of \cite{Fer(72)}, The process $\Set{J_1,J_2,\dots}$ is equal in law to the process $\Set{\bar{\rho}^{-1}(\Gamma_1),\bar{\rho}^{-1}(\Gamma_2),\dots}$ where $\bar{\rho}^{-1}(x) \coloneqq \big( \Gamma(1-\alpha)y \big)^{-1/\alpha}$ and $\Set{\Gamma_1,\Gamma_2,\dots}$ are the jumps of a standard homogeneous Poisson process on the half real-line. Deduce that
  \begin{align*}
    \sup_{x\in (0,1)}x^{-\beta}\Big| \frac{\bar{F}_P(x)}{L_{\alpha}(P)x^{-\alpha}} -1\Big|
    &\overset{d}{=} T^{\beta}\Gamma(1-\alpha)^{\beta/\alpha} \sup_{x\in (\bar{\rho}(T),\infty)} x^{-1 + \beta/\alpha}\Big| \sum_{j\geq 1}I(\Gamma_j \leq x) - x\Big|\\
    &\leq T^{\beta}\Gamma(1-\alpha)^{\beta/\alpha} \sup_{x\in (0,\infty)} x^{-1 + \beta/\alpha}\Big| \sum_{j\geq 1}I(\Gamma_j \leq x) - x\Big|.
  \end{align*}
  Since $T$ follows an $\alpha$-stable distribution under $\PP_{\alpha}$, for all $\varepsilon > 0$ we can find $M > 0$ such that $T^{\beta} \leq M$ with probability $\geq 1 - \varepsilon$. Then the result follows from Lemma~\ref{lem:poissoncdf} below.
 \end{proof}

\begin{lemma}
  \label{lem:poissoncdf}
  Let $\Set{\Gamma_1,\Gamma_2,\dots}$ be the jumps of a standard homogeneous Poisson process on the half real-line and let $0 < \delta < 1/2$. Then for all $\varepsilon > 0$ there exists a constant $B > 0$ such that with probability more than $1 - \varepsilon$
  \begin{equation*}
    \sup_{x\in (0,\infty)} x^{-(1/2 + \delta)}\Big| \sum_{j\geq 1}I(\Gamma_j \leq x) - x\Big| \leq B.
  \end{equation*}
\end{lemma}
\begin{proof}
  We only sketch the proof as it is a trivial adaptation of the proof in \cite[Lemma~B.1]{Fav(22)}. Defining $\Gamma_0 = 0$
  \begin{align*}
    \sup_{x\in (0,\infty)} x^{-(1/2 + \delta)}\Big| \sum_{j\geq 1}I(\Gamma_j \leq x) - x\Big|
    &= \sup_{k\geq 0}\sup_{x\in (\Gamma_k,\Gamma_{k+1}]} x^{-(1/2 + \delta)}\Big| \sum_{j\geq 1}I(\Gamma_j \leq x) - x\Big|\\
    &= \sup_{k\geq 0}\sup_{x\in (\Gamma_k,\Gamma_{k+1}]} x^{-(1/2 + \delta)}| k - x|\\
    &\leq \max\Big(\Gamma_1^{1/2 - \delta},\ \sup_{k\geq 1}\Gamma_k^{-(1/2+\delta)}\max(|\Gamma_{k} - k|,\,|\Gamma_{k+1} - k|) \Big)\\
    &\leq \max\Big(\Gamma_1^{1/2-\delta},\ \sup_{k\geq 1}\Gamma_k^{-(1/2+\delta)}(|\Gamma_{k} - k| + \xi_{k+1}) \Big).
  \end{align*}
  That is,
  \begin{equation*}
    \sup_{x\in (0,\infty)} x^{-(1/2 + \delta)}\Big| \sum_{j\geq 1}I(\Gamma_j \leq x) - x\Big|
    \leq \max\Bigg(\Gamma_1^{1/2-\delta},\ \sup_{k\geq 1}\Big(\frac{\Gamma_k}{k}\Big)^{-(1/2+\delta)}\cdot\sup_{k\geq 1}\frac{|\Gamma_{k} - k| + \xi_{k+1}}{k^{1/2+\delta}} \Bigg).
  \end{equation*}
  Then it suffices to show that the last rhs is bounded with high probability. Details can be found in \cite[Lemma~B.1]{Fav(22)}. A quick heuristic (non rigorous though) show that this must be true since $\lim_k\frac{\Gamma_k}{k} = 1$ almost-surely as $k\to \infty$ by the law of large numbers and $\limsup_k\frac{|\Gamma_k - k|}{\sqrt{2k\log(\log(k))}} = 1$ by the law of iterated logarithm.
\end{proof}

\subsection{Proof of Theorem~\ref{thm:rare-types-minimax}}
\label{sec:proof-theorem-rare-types-minimax}

\subsubsection{Proof of the upper bound}
\label{sec:proof-upper-bound}

  Let $P\in \Sigma(\alpha,\beta,C,C')$ arbitrary and let $\Omega_n$ be the event on which
  \begin{align*}
    \sqrt{n^{\alpha}L_{\alpha}(P)}\ell(\hat{\theta}_1^{\text{\tiny{(GT)}}},\theta_1)%
    &\leq 64e \sqrt{\frac{1 + \frac{3C'\Gamma(1-\alpha + \beta)}{\Gamma(1-\alpha)n^{\beta}}}{\Gamma(1-\alpha)}}\sqrt{x} +  \frac{16e(x + \log(n))}{\Gamma(1-\alpha)\sqrt{C}n^{\alpha/2}}x,
  \end{align*}
  and
  \begin{align*}
    \sqrt{n^{\alpha}L_{\alpha}(P)}\ell(\hat{\theta}_2^{\text{\tiny{(GT)}}},\theta_2)%
    &\leq 512e^2\sqrt{\frac{1+\frac{3C'\Gamma(2-\alpha + \beta)}{\Gamma(2-\alpha)n^{\beta}}}{\Gamma(2-\alpha)}}\sqrt{x} + \frac{64e^2(x+\log(n))^2}{\Gamma(2-\alpha)\sqrt{C}n^{\alpha/2}}x.
  \end{align*}
  By Theorem~\ref{thm:5} the event $\Omega_n$ occurs with probability at least $1-14e^{-x}$. It is easily checked that
  $\Gamma(1-\alpha) \geq 1$, $\frac{\Gamma(2-\alpha+\beta)/\Gamma(2-\alpha)}{\Gamma(1-\alpha+\beta)/\Gamma(1-\alpha)} = \frac{1-\alpha+\beta}{1-\alpha}\geq 1$, and $\Gamma(2-\alpha) \geq \inf_{z> 0}\Gamma(z) \geq \frac{1}{2}$ [the minimum of the gamma function is known to be approximately $0.88$]. Furthermore, letting $\psi(x) \coloneqq \log\Gamma(x)$, it is found by a Taylor expansion that $\frac{\Gamma(2-\alpha+\beta)}{\Gamma(2-\alpha)} = \exp(\psi'(2-\alpha + \kappa)\beta)$ for some $\kappa \in (0,\beta)$. But the function $\psi'$ is monotone and $\psi'(x) \leq \log(x)$ for all $x > 0$. Hence, $\frac{\Gamma(2-\alpha+\beta)}{\Gamma(2-\alpha)} \leq (2-\alpha+\beta)^{\beta}\leq (2+\beta)^{\beta} =\beta^{\beta}e^{\beta\log(1+2/\beta)}\leq e^2\beta^{\beta}$. Consequently, on the event $\Omega_n$
  \begin{align*}
    \sqrt{n^{\alpha}L_{\alpha}(P)}\ell(\hat{\theta}_1^{\text{\tiny{(GT)}}},\theta_1)
    &\leq 64e\sqrt{1 + 3e^2C'(\beta/n)^{\beta}}\sqrt{x} + \frac{16e(x + \log(n))}{\sqrt{C}n^{\alpha/2}}x,\\
    \sqrt{n^{\alpha}L_{\alpha}(P)}\ell(\hat{\theta}_2^{\text{\tiny{(GT)}}},\theta_2)%
    &\leq 512\sqrt{2}e^2\sqrt{1+3e^2C'(\beta/n)^{\beta}}\sqrt{x} + \frac{64\sqrt{2}e^2(x+\log(n))^2}{\sqrt{C}n^{\alpha/2}}x.
  \end{align*}
  In particular, whenever $x \leq \min\{(\frac{Cn^{\alpha}}{512\sqrt{2}e^2})^{1/3},\, \frac{Cn^{\alpha}}{512\sqrt{2}e^2\log(n)^2},\, \frac{Cn^{\alpha}}{2048^2\cdot 2e^4(1+3e^2C'(\beta/n)^{\beta})} \}$, on $\Omega_n$ it is the case that
  \begin{align*}
    \ell(\hat{\theta}_2^{\text{\tiny{(GT)}}},\theta_2)%
    \leq \frac{1}{2}.
  \end{align*}
  Next, remark that if $P \in \Sigma(\alpha,\beta,C,C')$ then $\theta_r(P;\mathbf{X}_n) \ne 0$ $P$-as. Consequently,
\begin{align*}
  \ell(\hat{T},T)%
  &=%
    \Big|\frac{\hat{\theta}_1^{\text{\tiny{(GT)}}}}{\hat{\theta}_2^{\text{\tiny{(GT)}}}}\frac{\theta_2}{\theta_1} - 1 \Big|\\%
  &\leq \Big|\frac{\hat{\theta}_1^{\text{\tiny{(GT)}}}}{\theta_1} - 1 \Big|%
    + \frac{\hat{T}}{T}\Big|\frac{\hat{\theta}_2^{\text{\tiny{(GT)}}}}{\theta_2} - 1 \Big|\\
  &\leq \ell(\hat{\theta}_1^{\text{\tiny{(GT)}}},\theta_1)%
    + \ell(\hat{\theta}_2^{\text{\tiny{(GT)}}},\theta_2)\Big(1 + \ell(\hat{T},T)\Big).
\end{align*}
Deduce that on the event where $\ell(\hat\theta_2,\theta_2) < 1$:
\begin{equation*}
  \ell(\hat{T},T)%
  \leq\frac{\ell(\hat{\theta}_1^{\text{\tiny{(GT)}}},\theta_1) + \ell(\hat{\theta}_2^{\text{\tiny{(GT)}}},\theta_2)}{1 - \ell(\hat{\theta}_2^{\text{\tiny{(GT)}}},\theta_2)}.
\end{equation*}
Hence the conclusion, after noticing that $2\cdot(512\sqrt{2}e^2 + 64 e) \leq 12000$, that $2\cdot(64\sqrt{2}e^2 + 16e) \leq 1500$, that $512\sqrt{2}e^2 \leq 6000$, and that $2048^2\cdot 2e^4 \leq 5\cdot 10^8$.

\subsubsection{Proof of the lower bound}
\label{sec:proof-lower-bound}

Mimicking the proof of Theorem~\ref{thm:5} up to equation~\eqref{eq:6}, it is found that
  \begin{multline*}
    \sup_{P\in \Sigma(\alpha,\beta,C,C')}\PP_P\Big( \sqrt{n^{\alpha}L_{\alpha}(P)} \ell(\hat{T},T(P;\mathbf{X}_n)) > AC^{1/2} \Big)\\%
    \geq \int_{\mathcal{P}}\PP_P\Big(n^{\alpha/2}\ell(\hat{T},T(P;\mathbf{X}_n)) > A \Big) \mathrm{PY}_{\alpha,d,G}(\intd P)%
    - \mathrm{PY}_{\alpha,0,G}\big( \Sigma(\alpha,\beta,C,C')^c \big).
  \end{multline*}
  The term $\mathrm{PY}_{\alpha,d,G}( \Sigma(\alpha,\beta,C,C')^c)$ has been taken care of in the proof of Theorem~\ref{thm:5}. In particular it has been demonstrated that it can be made smaller than $1/4$ by choosing $C,C'$ accordingly. The rest of the proof consists in showing that the first term can be made larger than $\frac{1}{2}$ by suitable choice of $A$. To deal with the first term, we mimick the proof of Theorem~\ref{thm:3} up to equation~\eqref{eq:11} to find that for all $\varepsilon > 0$
  \begin{align*}
    \int_{\mathcal{P}}\PP_P\Big( \ell(\hat{T},T(P;\mathbf{X}_n)) > \varepsilon \Big) \mathrm{PY}_{\alpha,0,G}(\intd P)%
    &\geq \EE_{\Pi}\Big[ \inf_{t\in \NNReals}\EE_{\Pi}\big( I(\ell(t,T(P;\mathbf{X}_n))> \varepsilon)  \mid \mathbf{X}_n\big) \Big].
  \end{align*}
  Here the proof needs some modification compared to Theorem~\ref{thm:3}. But, using Lemma~\ref{lem:4}, conditional on $\mathbf{X}_n$ it is the case that $(\theta_1(P;\mathbf{X}_n),\theta_2(P;\mathbf{X}_n))$ has the law of $(W_0,W_0^2\sum_{j\geq 1}Q_j^2)$ with $W_0 \mid \mathbf{X}_n \sim\mathrm{Beta}(d + \alpha K_n, n - \alpha K_n)$ is independent of $Q \mid \mathbf{X}_n \sim \mathrm{PY}_{\alpha,\alpha K_n,G}$. Consequently, $T(P;\mathbf{X}_n)$ is equal in law to $\frac{1}{W_0\sum_{j\geq 1}Q_j^2}$. In particular $T(P;\mathbf{X}_n)$ is almost-surely non-zero, so $\ell(t,T(P;\mathbf{X}_n)) = \big|\frac{t}{T(P;\mathbf{X}_n)} -1 \big|$ almost-surely too. Letting $Z = \sum_{j\geq 1}Q_j^2$, and following the steps up to equation~\eqref{eq:11}, it is found that
  \begin{align*}
    \inf_{t\in \NNReals}\EE_{\Pi}\big(I( \ell(t,T(P;\mathbf{X}_n)) > \varepsilon\} ) \mid \mathbf{X}_n\big)
    &\geq \inf_{t\in \NNReals}\inf_{z\in \NNReals} \EE_{\Pi}\big(I( \{ |t z W_0 - 1| > \varepsilon\} ) \mid \mathbf{X}_n\big)\\
    &=1 - \sup_{t\in \NNReals}\PP_{\Pi}\Big(1 - \varepsilon \leq t W_0 \leq 1 + \varepsilon \mid \mathbf{X}_n\Big)\\
    &= 1 - \sup_{t\in \NNReals}\PP_{\Pi}\Big( t(1 - \varepsilon) \leq W_0 \leq t(1 + \varepsilon) \mid \mathbf{X}_n\Big)\\
  \end{align*}
  We deduce from Lemma~\ref{lem:beta-bayes} with $a \equiv \alpha K_n$, $b \equiv n - \alpha K_n$, $\delta_1 \equiv 1-\varepsilon$, and $\delta_2 \equiv 1+\varepsilon$, that on the event where $\alpha K_n < n-1$ [recall $n\geq 4$ by assumption]
  \begin{align*}
    \sup_{t\in \NNReals}\PP_{\Pi}\Big( t(1 - \varepsilon) \leq W_0 \leq t(1 + \varepsilon) \mid \mathbf{X}_n\Big)%
    &\leq \Bigg(\frac{1+\varepsilon}{1-\varepsilon} - 1\Bigg)\sqrt{\frac{\alpha K_n(n-1)}{2\pi(n-\alpha K_n-1}}e^{\frac{1}{12(n-1)}}\\
    &\leq \frac{\varepsilon}{1-\varepsilon}\sqrt{\frac{\alpha K_n(n-1)}{n-\alpha K_n-1}}
  \end{align*}
  Therefore
  \begin{align*}
    &\int_{\mathcal{P}}\PP_P\Big(\ell(\hat{T},T(P;\mathbf{X}_n)) > \varepsilon \Big) \mathrm{PY}_{\alpha,d,G}(\intd P)\\
    &\qquad\qquad\geq%
      \EE_{\Pi}\Big[\inf_{t\in \NNReals}\EE_{\Pi}\big(I( \ell(t,T(P;\mathbf{X}_n)) > \varepsilon\} ) \mid \mathbf{X}_n\big)I\big(\alpha K_n\leq \min(n/2-1,3 \EE_{\Pi}(\alpha K_n)) \big)\Big]\\
    &\qquad\qquad\geq%
      \Bigg(1 - \frac{\varepsilon\sqrt{6 \EE_{\Pi}(\alpha K_n)}}{1 - \varepsilon}\Bigg)\PP_{\Pi}\big(\alpha K_n \leq \min(n/2-1,\, 3\EE_{\Pi}(\alpha K_n)) \Big)\\
    &\qquad\qquad\geq \Bigg(1 - \frac{\varepsilon\sqrt{6 \alpha n^{\alpha}}}{1 - \varepsilon}\Bigg)\Bigg(1 - \max\Big(\frac{1}{3},\,\frac{4\alpha n^{\alpha}}{n} \Big)\Bigg)
  \end{align*}%
  where the last line follows by the exact same steps that lead us to \eqref{eq:50:final} in the proof of Theorem~\ref{thm:5}. The conclusion follows by taking $\varepsilon = An^{-\alpha/2}$ with $A = \min(\frac{1}{8\sqrt{6\alpha}},\frac{n^{\alpha/2}}{2})$.


\section*{Acknowledgements}

The authors are grateful to Giulia Cereda and Richard D. Gill for stimulating conversations on the rare type match problem and its interplay with the estimation of the missing mass. Stefano Favaro and Zacharie Naulet received funding from the European Research Council (ERC) under the European Union's Horizon 2020 research and innovation programme under grant agreement No 817257. Stefano Favaro gratefully acknowledge the financial support from the Italian Ministry of Education, University and Research (MIUR), ``Dipartimenti di Eccellenza" grant 2018-2022.


\end{document}